\documentclass[12pt,a4paper,oneside,english]{amsart} 
\usepackage{amsmath, amssymb, amscd} 
\usepackage{mathtools}
\usepackage[DIV12]{typearea} 
\usepackage{babel} 
\usepackage[latin1]{inputenc} 
\usepackage[T1]{fontenc} 
\usepackage{lmodern} 
\usepackage{textcomp} 
\usepackage{enumerate}

\DeclarePairedDelimiter\abs{\lvert}{\rvert}

\theoremstyle{plain} 
\newtheorem{theorem}{Theorem}[section] 
\newtheorem{lemma}[theorem]{Lemma} 
\newtheorem{prop}[theorem]{Proposition} 
 
\newtheorem{cor}[theorem]{Corollary}

\theoremstyle{remark} 
\newtheorem{remark}[theorem]{Remark}

\theoremstyle{definition} 
\newtheorem{definition}[theorem]{Definition}

\numberwithin{equation}{section}

\title[Simple Cuntz-Pimsner rings]{Simple Cuntz-Pimsner rings}

\author{Toke Meier Carlsen} 
\address{Department of Mathematical Sciences\\
	NTNU\\
	NO-7491 Trondheim\\
	Norway } 
\email{Toke.Meier.Carlsen@math.ntnu.no} 

\author{Eduard Ortega} 
\address{Department of Mathematical Sciences\\
	NTNU\\
	NO-7491 Trondheim\\
	Norway }  
\email{eduardor@math.ntnu.no} 

\author{Enrique Pardo} 
\address{Departamento de Matem\'aticas, Facultad de Ciencias\\
	Universidad de C\'adiz, Campus de Puerto Real\\
	11510 Puerto Real (C\'adiz)\\
	Spain.} 
\email{enrique.pardo@uca.es}
\urladdr{https://sites.google.com/a/gm.uca.es/enrique-pardo-s-home-page/} 

\thanks{This research was supported by the NordForsk Research Network "Operator Algebras and Dynamics"" (grant \#11580). The first named author was partly supported by The Danish Natural Science Research Council and the Research Council of Norway, and the second named author was partially supported by MEC-DGESIC (Spain) through Project MTM2008-06201-C02-01/MTM. The third author was partially supported by the DGI and European Regional Development Fund, jointly, through Project MTM2008-06201-C02-02 and by PAI III grants FQM-298 and P07-FQM-02467 of the Junta de Andaluc\'{\i}a. The second and third authors were partially supported by the Consolider Ingenio "Mathematica" project CSD2006-32 by the MEC and by 2009 SGR 1389 grant of the Comissionat per Universitats i Recerca de la Generalitat de Catalunya.} 
\date{\today}

\renewcommand{\implies}{\Rightarrow}

\newcommand{\OCK}{\mathcal{O}_{(P,Q,\psi)}(J)} 
\newcommand{\FCK}[1][0]{\OCK^{(#1)}} 
\newcommand{\Z}{\mathbb{Z}} 
\newcommand{\N}{\mathbb{N}} 
\newcommand{\No}{{\mathbb{N}_{0}}} 
\newcommand{\toeplitz}{\mathcal{T}_{(P,Q,\psi)}}

\DeclareMathOperator{\spn}{span} \DeclareMathOperator{\id}{Id}

\begin{document}
\begin{abstract}
	Necessary and sufficient conditions for when every non-zero ideal in a relative Cuntz-Pimsner ring contains a non-zero graded ideal, when a relative Cuntz-Pimsner ring is simple, and when every ideal in a relative Cuntz-Pimsner ring is graded, are given. A ``Cuntz-Krieger uniqueness theorem'' for relative Cuntz-Pimsner rings is also given and condition (L) and condition (K) for relative Cuntz-Pimsner rings are introduced. 
\end{abstract}

\maketitle

\section{Introduction}

In \cite{CaOr2011} the two first named authors introduced the notion of a relative Cuntz-Pimsner ring $\OCK$ as an algebraic analogue of (relative) Cuntz-Pimsner $C^*$-algebras (see for example \cite{MuSo1998}, \cite{Pi1997}, \cite{Ka2004a} and \cite{Ka2007}), and showed that for instance Leavitt path algebras (see for example \cite{AbAr2005}, \cite{AbAr2008} and \cite{To2007}), crossed products of a ring by a single automorphism (also called a skew group ring, see for example \cite{Mo1980} and \cite{Pa1989}) and fractional skew monoid rings of a single corner isomorphism (see \cite{ArGoGo2004}) can be constructed as relative Cuntz-Pimsner rings. They also gave a complete description of the graded ideals of an arbitrary relative Cuntz-Pimsner ring $\mathcal{O}_{(P,Q,\psi)}(J)$. The purpose of this paper is to study the non-graded ideals of such a relative Cuntz-Pimsner ring $\OCK$. Although we do not reach a complete description of all (graded or non-graded) ideals of $\OCK$, we do find necessary and sufficient conditions for when every non-zero ideal in $\OCK$ contains a non-zero graded ideal (Theorem \ref{thm:L}), when $\OCK$ is simple (Theorem \ref{thm:simple}), and when every ideal in $\OCK$ is graded (Theorem \ref{thm:K}). We also give a ``Cuntz-Krieger uniqueness theorem'' for $\OCK$ (Theorem \ref{thm:CK}) and introduce condition (L) (Definition \ref{def:L}) and condition (K) (Definition \ref{def:K}) for relative Cuntz-Pimsner rings. These results and definitions are generalizations of similar results and definitions about Lea\-vitt path algebras given in \cite{To2007}, and analogues of similar results and definitions given in the $C^*$-algebraic setting for graph $C^*$-algebras (see for example \cite{Ra2005}), ultragraph $C^*$-algebras (see \cite{To2003d}), topological graph $C^*$-algebras (see \cite{Ka2006b}), and (relative) Cuntz-Krieger algebras of finitely aligned higher rank graphs (see for example \cite{Si2006}).

It is worth pointing out that analogues in the $C^*$-algebraic setting of these results do not exist in the generality of this paper. It does not seem unreasonable to believe that it should be possible to obtain such, but a different approach than the one used in this paper seems to be needed.

\subsection*{Contents}

Section \ref{sec:preliminaries} contains some preliminary results and the pivotal Proposition \ref{prop:main}. In Section \ref{sec:condition-L} condition (L) is introduced (Definition \ref{def:L}), and sufficient and necessary conditions for when every non-zero ideal in $\OCK$ contains a non-zero graded ideal are given (Theorem \ref{thm:L}). Section \ref{sec:cuntz-krieg-uniq} contains the Cuntz-Krieger uniqueness theorem (Theorem \ref{thm:CK}). In Section \ref{sec:simplicity-ock} sufficient and necessary for when $\OCK$ is simple are given (Theorem \ref{thm:simple}), and in Section \ref{sec:condition-k} condition (K) is introduced (Definition \ref{def:K}), and sufficient and necessary conditions for when every ideal in $\OCK$ is graded are given (Theorem \ref{thm:K}). In Section \ref{sec:toeplitz-rings} the case when $J=0$ and $\OCK$ is the Toeplitz ring $\toeplitz$ of $(P,Q,\psi)$ is considered. Finally, in Section \ref{sec:leavitt} and Section \ref{sec:cross} we illustrate the results obtained in the paper by applying them to Leavitt path algebras (Section \ref{sec:leavitt}), and to crossed products of a ring by a single automorphism and fractional skew monoid rings of a single corner isomorphism (Section \ref{sec:cross}), and thereby obtain characterizations of when these algebras are simple. The characterization of when a Leavitt path algebra is simple is well-know (see \cite[Theorem 6.18]{To2007}), whereas the characterizations of when a crossed product of a ring by an automorphism and a fractional skew monoid ring by a corner isomorphism are simple, to the knowledge of the authors, are new.

\subsection*{Notation and conventions}

In this paper every ideal will be a two-sided ideal. The set of integers will be denoted by $\Z$, the set of positive integers will be denoted by $\N$ and the set of non-negative integers will be denoted by $\No$.

We will use the same notation as in \cite{CaOr2011} with the addition that $R$ will always denote a fixed ring, $(P,Q,\psi)$ will be a fixed $R$-system satisfying condition (FS) and $J$ will be a fixed faithful and $\psi$-compatible ideal in $R$. To ease notation we will let $\sigma$, $S$, $T$ and $\pi$ denote $\iota_R^J$, $\iota_P^J$, $\iota_Q^J$ and $\pi^J$, respectively. We will repeatedly use that since $(P,Q,\psi)$ satisfies condition (FS), the $R$-system $(P^{\otimes n},Q^{\otimes n},\psi_n)$ will for each $n\in\N$ also satisfy condition (FS) (see \cite[Lemma 3.8]{CaOr2011}).

% Recall that $\OCK$ comes with a $\Z$-grading
% $\bigoplus_{n\in\Z}\FCK[n]$. Let $\pi$ denote the corresponding
% projection from $\OCK$ to $\FCK$.
\section{Preliminaries} \label{sec:preliminaries}

This section contains some preliminary results leading to Proposition \ref{prop:main}, which is pivotal for the rest of the paper.
\begin{lemma}
	\label{lemma:1} If $n\in\N$, $x_{-n}\in\FCK[-n]\setminus\{0\}$ and $x_n\in\FCK[n]\setminus\{0\}$, then there is a $p\in P^{\otimes n}$ and a $q\in Q^{\otimes n}$ such that $x_{-n}T^n(q)\ne 0$ and $S^n(p)x_n\ne 0$. 
\end{lemma}
\begin{proof}
	Write $x_n$ as $\sum_{i=1}^kT^n(q_i)y_i$ where $q_i\in Q^{\otimes n}$ and $y_i\in\FCK$ for $i=1,2,\dots, k$. It follows from condition (FS) that there is a $\theta\in\mathcal{F}_{P^{\otimes n}}(Q^{\otimes n})$ such that $\theta q_i=q_i$ for each $i=1,2,\dots,k$. It follows that $S^n(p)x_n$ cannot be 0 for all $p\in P^{\otimes n}$. That $x_{-n}T^n(q)\ne 0$ for some $q\in Q^{\otimes n}$ can be proved in a similar way. 
\end{proof}
\begin{definition}
	For an ideal $I$ in $R$, let $\psi^{-1}(I)$ be the ideal 
	\begin{equation*}
		\bigl\{x\in R\mid \psi(px\otimes q)\in I\text{ for all }q\in Q\text{ and all }p\in P\bigr\}, 
	\end{equation*}
	and let $I^{[\infty]}$ be the ideal 
	\begin{equation*}
		\bigcap_{k=1}^\infty I^{[k]} 
	\end{equation*}
	where $I^{[k]}$ is defined recursively by $I^{[1]}=I$ and $I^{[k]}=\psi^{-1}\bigl(I^{[k-1]}\bigr)\cap I$ for $k> 1$. 
\end{definition}

Recall that if $I$ is an ideal in $R$, then $QI=\spn\{qx\mid q\in Q,\ x\in I\}$ (see \cite[Definition 7.1]{CaOr2011}).
\begin{lemma}
	\label{lemma:2} Let $x\in R$. Then $x\in\psi^{-1}(I)$ if and only if $xq\in QI$ for all $q\in Q$. 
\end{lemma}
\begin{proof}
	Assume first that $x\in\psi^{-1}(I)$ and that $q\in Q$. Then it follows from condition (FS) that there are $q_1,\dots,q_m\in Q$ and $p_1,\dots,p_m\in P$ such that $xq=\sum_{i=1}^mq_i\psi(p_i\otimes xq)$. Since each $\psi(p_i\otimes xq)\in I$, it follows that $xq\in QI$.
	
	Assume then that $x\in R$ and $xq\in QI$ for all $q\in Q$, and let $q\in Q$ and $p\in P$. Then there are $q_1,\dots q_m\in Q$ and $x_1,\dots,x_m\in I$ such that $xq=\sum_{i=1}^mq_ix_i$, from which it follows that $\psi(px\otimes q)=\psi(p\otimes xq)=\sum_{i=1}^m\psi(p\otimes q_i)x_i\in I$. Thus $x\in \psi^{-1}(I)$.
\end{proof}

Let us now specialise to the case where $I=J$.
\begin{lemma}
	\label{lemma:3} Let $k\in\N$ and $x\in R$. Then $x\in J^{[k]}$ if and only if $\sigma(x)\in\spn\{T^k(q)S^k(p)\mid q\in Q^{\otimes k}, p\in P^{\otimes k}\}$. 
\end{lemma}
\begin{proof}
	We will prove the lemma by induction over $k$. For $k=1$ the lemma follows from \cite[Proposition 3.28]{CaOr2011}.
	
	Assume now that $k>1$ and that $x\in J^{[k-1]}$ if and only if $\sigma(x)\in\spn\{T^{k-1}(q)S^{k-1}(p)\mid q\in Q^{\otimes k-1}, p\in P^{\otimes k-1}\}$. We will then prove that $x\in J^{[k]}$ if and only if $\sigma(x)\in\spn\{T^k(q)S^k(p)\mid q\in Q^{\otimes k},\ p\in P^{\otimes k}\}$ for all $x\in R$. If $x\in J^{[k]}=\psi^{-1}(J^{[k-1]})\cap J$, then it follows from \cite[Proposition 3.28]{CaOr2011} that there are $q_1,\dots,q_m\in Q$ and $p_1,\dots,p_m\in P$ such that $\sigma(x)=\sum_{i=1}^mT(q_i)S(p_i)$. It follows from condition (FS) that there are $q'_1,\dots,q'_n\in Q$ and $p'_1,\dots,p'_n\in P$ such that $\sum_{j=1}^n\theta_{p'_j,q'_j}p_i=p_i$ for each $i$, from which it follows that 
	\begin{equation*}
		\sigma(x)=\sum_{i=1}^mT(q_i)S(p_i)=\sum_{i=1}^m\sum_{j=1}^nT(q_i)S(p_i)T(q'_j)S(p'_j) =\sum_{j=1}^nT(xq'_j)S(p'_j). 
	\end{equation*}
	It follows from Lemma \ref{lemma:2} that there for each $j$ are $q_{j,1}\dots,q_{j,m_j}\in Q$ and $x_{j,1},\dots,x_{j,m_j}\in J^{[k-1]}$ such that $xq_j=\sum_{l=1}^{m_j}q_{j,l}x_{j,l}$, and it then follows from the induction hypothesis that 
	\begin{equation*}
		\begin{split}
			\sigma(x)&=\sum_{j=1}^nT(xq'_j)S(p'_j)\\
			&=\sum_{j=1}^n\sum_{l=1}^{m_j}T(q_{j,l})\sigma(x_{j,l})S(p'_j)\in \spn\{T^k(q)S^k(p)\mid q\in Q^{\otimes k}, p\in P^{\otimes k}\}. 
		\end{split}
	\end{equation*}
	
	Conversely, if $\sigma(x)=\sum_{i=1}^mT^k(q_i)S^k(p_i)$, then $\iota_R(x)-\sum_{i=1}^m\iota_Q^k(q_i)\iota_P^k(p_i)\in\mathcal{T}(J)$ (cf. \cite[Definition 3.15 and 3.16]{CaOr2011}), so it follows from  \cite[Lemma 3.21]{CaOr2011} that $x\in J$. If $p\in P$ and $q\in Q$, then 
	\begin{multline*}
		\sigma\bigl(\psi(px\otimes q)\bigr)=S(p) \sum_{i=1}^mT^k(q_i)S^k(p_i) T(q)\\\in \spn\{T^{k-1}(q')S^{k-1}(p')\mid q'\in Q^{\otimes k-1}, p'\in P^{\otimes k-1}\}, 
	\end{multline*}
	which together with the induction hypothesis implies that $\psi(px\otimes q)\in J^{[k-1]}$, and thus that $x\in \psi^{-1}(J^{[k-1]})\cap J=J^{[k]}$. 
\end{proof}
\begin{definition}
	A subring $A$ of $\mathcal{O}_{(P,Q,\psi)}(J)$ has the \emph{ideal intersection property} if the implication $K\cap A=\{0\}\implies K=\{0\}$ holds for every ideal $K$ in $\OCK$. 
\end{definition}

We of course have that $\OCK$ itself has the ideal intersection property. We will in this paper study when $\sigma(R)$ and $\FCK$ have the ideal intersection property. We begin with $\FCK$.

Let $n\in\N$. Recall from \cite[Section 2]{CaOr2011} that there for each $p\in P$ exists a unique $R$-bimodule homomorphism $S_p:Q^{\otimes n+1}\to Q^{\otimes n}$ characterised by $S_p(q\otimes q_n)=\psi(p\otimes q)q_n$ for $q\in Q$ and $q_n\in Q^{\otimes n}$. Similarly, there exists for each $q_n\in Q^{\otimes n}$ an $R$-bimodule homomorphism $T_{q_n}:Q\to Q^{\otimes n+1}$ given by $T_{q_n}(q)=q_n\otimes q$ for $q\in Q$. Notice that $T^n(S_pT_{q_n}(q))=S(p)T^n(q_n)T(q)$ for $p\in P$, $q_n\in Q^{\otimes n}$ and $q\in Q$.

\begin{prop}
	\label{prop:main} The following 3 conditions are equivalent: 
	\begin{enumerate}
		\item \label{item:1} The subring $\FCK$ does not have the ideal intersection property. 
		\item \label{item:2} There is a non-zero graded ideal $\bigoplus_{k\in\Z}H^{(k)}$ in $\OCK$, an $n\in\N$ and a family $(\phi_k)_{k\in\Z}$ of injective $\FCK$-bimodule homomorphisms $\phi_k:H^{(k)}\to\FCK[k+n]$ such that $x\phi_k(y)=\phi_{k+j}(xy)$ and $\phi_k(y)x=\phi_{k+j}(yx)$ for $k,j\in\Z$, $x\in\FCK[j]$ and $y\in H^{(k)}$. 
		\item \label{item:3} There is a non-zero $\psi$-invariant ideal $I_0$ of $R$, an $n\in\N$ and an injective $R$-bimodule homomorphism $\eta:I_0\to Q^{\otimes n}$ such that $S_pT_{\eta(x)}(q)=\eta(\psi(px\otimes q))$ for $p\in P$, $x\in I_0$ and $q\in Q$, and such that $I_0\subseteq J^{[\infty]}$.
	\end{enumerate}
\end{prop}
\begin{proof}
	$\eqref{item:1}\implies\eqref{item:2}$: Let $K$ be a non-zero ideal in $\OCK$ such that $K\cap\FCK=\{0\}$. Let $N$ be the set of $n\in\N_0$ for which there are $x_i\in\FCK[i]$, $i=0,1,\dots,n$ with $x_0\ne 0$ such that $\sum_{i=0}^nx_i\in K$. Let $\sum_{i=j}^kx_i\in K$ where $j\le k\in\Z$, $x_i\in\FCK[i]$ for $i=j,j+1,\dots,k$ and $x_j\ne 0$. If $j\ne 0$, then it follows from Lemma \ref{lemma:1} that there is a $y_{-j}\in\FCK[-j]$ such that either $y_{-j}x_j$ or $x_jy_{-j}$ is non-zero. It follows that $N\ne\emptyset$. Since $K\cap\FCK=\{0\}$, it follows that $0\notin N$. Let $n=\min N$. Then $n\in\N$.
	
	For each $k\in\Z$ let 
	\begin{equation*}
		H^{(k)}:=\left\{x_k\in\FCK[k]\Bigm| \exists x_{k+i}\in\FCK[k+i],\ i=1,2,\dots,n:\sum_{i=0}^nx_{k+i}\in K\right\}. 
	\end{equation*}
	If $x\in H^{(k)}$ and $y\in\FCK[j]$, then $xy,yx\in H^{(k+j)}$. It follows that $\bigoplus_{k\in\Z}H^{(k)}$ is a graded ideal in $\OCK$, and since $H^{(0)}\ne\{0\}$, it must be the case that $\bigoplus_{k\in\Z}H^{(k)}$ is non-zero.
	
	Let $k\in\Z$ and let $x_k\in H^{(k)}$. It follows from Lemma \ref{lemma:1} and the minimality of $n$ that there is a unique $x_{k+n}\in\FCK[k+n]$ satisfying that there exist $x_{k+i}\in\FCK[k+i]$, $i=1,2,\dots,n-1$ such that $\sum_{i=0}^nx_{k+i}\in K$. It also follows from Lemma \ref{lemma:1} and the minimality of $n$ that $x_{k+n}\ne 0$ if $x_k\ne 0$. Thus there is an injective map $\phi_k:H^{(k)}\to\FCK[k+n]$ sending $x_k$ to $x_{k+n}$. It is easy to check that $\phi_k$ is a $\FCK$-bimodule homomorphism, and that $x\phi_k(y)=\phi_{k+j}(xy)$ and $\phi_k(y)x=\phi_{k+j}(yx)$ when $k,j\in\Z$, $x\in\FCK[j]$ and $y\in H^{(k)}$.
	
	$\eqref{item:2}\implies\eqref{item:3}$: We will first prove that $H^{(0)}\cap\sigma(R)\ne\{0\}$, so assume, for contradiction, that $H^{(0)}\cap\sigma(R)=\{0\}$. Then it follows from \cite[Lemma 3.21 and Theorem 7.27]{CaOr2011} that 
	\begin{multline*}
		H^{(0)}=\spn\bigl(\{T^n(q)(\sigma(x)-\pi(\Delta(x)))S^n(p)\mid n\in\N,\ q\in Q^{\otimes n},\ x\in J',\ p\in P^{\otimes n}\}\\\cup\{\sigma(x)-\pi(\Delta(x))\mid x\in J'\}\bigr) 
	\end{multline*}
	for some faithful $\psi$-compatible ideal $J'$ of $R$ which contains $J$. We claim that $H^{(0)}$ must contain a non-zero element of the form $\sigma(x)-\pi(\Delta(x))$, $x\in J'$. To see that this is the case, let $y$ be a non-zero element of $H^{(0)}$ and write it as 
	\begin{equation*}
		\sigma(x_0)-\pi(\Delta(x_0))+\sum_{i=1}^kT^{n_i}(q_i)\bigl(\sigma(x_i)-\pi(\Delta(x_i))\bigr)S^{n_i}(p_i) 
	\end{equation*}
	where $k\in\N$, $x_0,x_1,\dots,x_k\in J'$ and $n_i\in\N$, $q_i\in Q^{\otimes n_i}$, $p_i\in P^{\otimes n_i}$ for each $i\in\{1,2,\dots,k\}$, and assume that $\sum_{i\in M} T^{n_i}(q_i)\bigl(\sigma(x_i)-\pi(\Delta(x_i))\bigr)S^{n_i}(p_i)\ne 0$ where $M$ is the set of those $i$'s for which $n_i$ is maximal among $\{n_1,n_2,\dots,n_k\}$. Let $n$ be the maximal value of $n_i$. It follows from condition (FS) that there are $q\in Q^{\otimes n}$ and $p\in P^{\otimes n}$ such that if we let $x=\sum_{i\in M}\psi_n(p\otimes q_i)x_i\psi_n(p_i\otimes q)$, then 
	\begin{equation*}
		\sigma(x)-\pi(\Delta(x))=S^n(p) \sum_{i\in M} T^{n_i}(q_i)\bigl(\sigma(x_i)-\pi(\Delta(x_i))\bigr)S^{n_i}(p_i)T^n(q)\ne 0. 
	\end{equation*}
	Since $(\sigma(x_0)-\pi(\Delta(x_0)))T^n(q)=0$ and $(\sigma(x_i)-\pi(\Delta(x_i)))S^{n_i}(p_i)T^n(q)=0$ for each $i\notin M$, it follows that 
	\begin{multline*}
		\sigma(x)-\pi(\Delta(x))\\=S^n(p)\biggl( \sigma(x_0)-\pi(\Delta(x_0))+\sum_{i=1}^kT^{n_i}(q_i)\bigl(\sigma(x_i)-\pi(\Delta(x_i))\bigr)S^{n_i}(p_i)\biggr)T^n(q)\in H^{(0)}. 
	\end{multline*}
	Thus $H^{(0)}$ contains a non-zero element of the form $\sigma(x)-\pi(\Delta(x))$, $x\in J'$. If follows from condition (FS) that there is a $p'\in P^{\otimes n}$ such that 
	\begin{equation*}
		S^n(p')\phi_0\bigl(\sigma(x)-\pi(\Delta(x)\bigr)\ne 0, 
	\end{equation*}
	but since $S^n(p'')(\sigma(x)-\pi(\Delta(x)))=0$ for all $p''\in P^{\otimes n}$, it follows that 
	\begin{equation*}
	S^n(p')\phi_0\bigl(\sigma(x)-\pi(\Delta(x)\bigr)=\phi_{-n}\bigl(S^n(p')(\sigma(x)-\pi(\Delta(x))\bigr)=0,
	\end{equation*}
	and we have reached a contradiction. Thus it must be the case that $H^{(0)}\cap\sigma(R)\ne\{0\}$. Let $I=\{x\in R\mid \sigma(x)\in H^{(0)}\}$. Then $I$ is a non-zero $\psi$-invariant ideal of $R$. For each $m\in\No$ let 
	\begin{equation*}
		A_m=\spn\bigl\{T^{n+k}(q)S^k(p)\mid k\in\{0,1,\dots,m\},\ q\in Q^{\otimes n+k},\ p\in P^{\otimes k}\bigr\} \subseteq \FCK[n]
	\end{equation*}
	and 
	\begin{equation*}
		I_m=\{x\in I\mid \phi_0(\sigma(x))\in A_m\}. 
	\end{equation*}
	Then $I_0\subseteq I_1\subseteq I_2\subseteq\dots$ and each $I_m$ is a $\psi$-invariant two-sided ideal in $R$. In fact, $x\in I_{m+1}$, implies that $\psi(px\otimes q)\in I_m$ for all $p\in P$ and $q\in Q$. Since $I$ is nonzero, there exists an $x\ne 0$ and an $m\in\No$ such that $x\in I_m$. Choose $k\in\N$ such that $kn\ge m$. Then 
	\begin{equation*}
		\phi_{(k-1)n}\circ\phi_{(k-2)n}\circ\dots\circ\phi_n\circ\phi_0(\sigma(x))\in\FCK[nk]\setminus\{0\} 
	\end{equation*}
	so it follows from Lemma \ref{lemma:1} that there is a $p\in P^{\otimes nk}$ such that 
	\begin{equation*}
		\phi_{-n}\circ\phi_{-2n}\circ\dots\circ\phi_{-(k-1)n}\circ\phi_{-kn}(S^{nk}(px))= S^{nk}(p) \phi_{(k-1)n}\circ\phi_{(k-2)n}\circ\dots\circ\phi_n\circ\phi_0(\sigma(x))\ne 0, 
	\end{equation*}
	from which it follows that $px\ne 0$. It follows from condition (FS) that there is a $q\in Q^{\otimes kn}$ such that $\psi_{kn}(px\otimes q)\ne 0$. We have that $\psi_{kn}(px\otimes q)\in I_0$, so $I_0\ne \{0\}$.
	
	Since $\phi_0(\sigma(x))\in T^n(Q^{\otimes n})$ for every $x\in I_0$, and $T^n:Q^{\otimes n}\to\FCK[n]$ is injective, we can define $\eta:I_0\to Q^{\otimes n}$ by, for $x\in I_0$, letting $\eta(x)$ be the unique element of $Q^{\otimes n}$ such that $T^n(\eta(x))=\phi_0(\sigma(x))$. It is straightforward to check that $\eta$ is an injective $R$-bimodule homomorphism, and if $p\in P$, $x\in I_0$ and $q\in Q$, then 
	\begin{equation*}
		\begin{split}
			T^n\Bigl(\eta\bigl(\psi(px\otimes q)\bigr)\Bigr)&=\phi_0\Bigl(\sigma\bigl(\psi(px\otimes q)\bigr)\Bigr) =S(p)\phi_0\bigl(\sigma(x)\bigr)T(q)\\
			&=S(p)T^n\bigl(\eta(x)\bigr)T(q)=T^n\bigl(S_pT_{\eta(x)}(q)\bigr), 
		\end{split}
	\end{equation*}
	from which it follows that $\eta(\psi(px\otimes q))=S_pT_{\eta(x)}(q)$.
	
	If $x\in I_0$ then it follows from condition (FS) that there are $q_i\in Q^{(n)}$, $p_i\in P^{(n)}$, $i=1,2,\dots,m$ such that $\sum_{i=0}^m\theta_{q_i,p_i}\eta(x)=\eta(x)$. We then have that 
	\begin{equation*}
		T^n(\eta(x))=T^n\left(\sum_{i=0}^m\theta_{q_i,p_i}\eta(x)\right) =\sum_{i=0}^mT^n(q_i)S^n(p_i)\phi_0(\sigma(x)) =\phi_0\left(\sum_{i=0}^mT^n(q_i)S^n(p_ix)\right) 
	\end{equation*}
	from which it follow that $\sigma(x)=\sum_{i=0}^m T^n(q_i)S^n(p_ix)$. It now follows from Lemma \ref{lemma:3} that $x\in J^{[n]}\subseteq J$. Since $I_0$ is $\psi$-invariant, it follows that $x\in J^{[\infty]}$.
	
	$\eqref{item:3}\implies\eqref{item:1}$: Let $K$ be the ideal in $\OCK$ generated by $\{\sigma(x)-T^n(\eta(x))\mid x\in I_0\}$. Clearly, $K$ is non-zero, so we just have to prove that $K\cap\FCK=\{0\}$. Using condition (FS) and the properties of $\eta$, one can show that if $p\in P$, $x\in I_0$ and $q\in Q$, then 
	\begin{equation*}
		S(p)\bigl(\sigma(x)-T^n(\eta(x))\bigr)\in\spn\bigl\{\bigl(\sigma(x')-T^n(\eta(x'))\bigr)S(p')\bigm| x'\in I_0,\ p'\in P\bigr\} 
	\end{equation*}
	and 
	\begin{equation*}
		\bigl(\sigma(x)-T^n(\eta(x))\bigr)T(q)\in\spn\bigl\{T(q')(\sigma(x')-T^n(\eta(x')))\bigm| q'\in Q,\ x'\in I_0\bigr\}. 
	\end{equation*}
	It follows that 
	\begin{equation*}
		\begin{split}
			K=\spn\Bigl(&\bigl\{T^k(q)\bigl(\sigma(x)-T^n(\eta(x))\bigr)\bigm| k\in\N,\ q\in Q^{\otimes k},\ x\in I_0\bigr\}\\
			&\cup \bigl\{T^k(q)\bigl(\sigma(x)-T^n(\eta(x))\bigr)S^l(p)\bigm| k,l\in\N,\ q\in Q^{\otimes k},\ x\in I_0,\ p\in P^{\otimes l}\bigr\}\\
			&\cup\bigl\{\sigma(x)-T^n(\eta(x))\bigm| x\in I_0\bigr\}\\
			&\cup \bigl\{T^k(q)\bigl(\sigma(x)-T^n(\eta(x))\bigr)\bigm| l\in\N,\ x\in I_0,\ p\in P^{\otimes l}\bigr\}\Bigr), 
		\end{split}
	\end{equation*}
	so to show that $K\cap\FCK=\{0\}$, it sufficies to show the following 3 things: 
	\begin{enumerate}
		[(i)] 
		\item \label{item:4} if $l\in\N$, $A$ is a finite subset of $\{(k,q,x,p)\mid k\in\N,\ q\in Q^{\otimes l+k},\ x\in I_0,\ p\in P^{\otimes k}\}$ and $B$ is a finite subset of $\{(q,x)\mid q\in Q^{\otimes l},\ x\in I_0\}$, then 
		\begin{equation*}
			\sum_{(k,q,x,p)\in A}T^{l+k}(q)\sigma(x)S^k(p) + \sum_{(q,x)\in B}T^l(q)\sigma(x)=0 
		\end{equation*}
		if and only if 
		\begin{equation*}
			\sum_{(k,q,x,p)\in A}T^{l+k}(q)T^n(\eta(x))S^k(p) + \sum_{(q,x)\in B}T^l(q)T^n(\eta(x))=0, 
		\end{equation*}
		\item \label{item:12} if $A$ is a finite subset of $\{(k,q,x,p)\mid k\in\N,\ q\in Q^{\otimes k},\ x\in I_0,\ p\in P^{\otimes k}\}$ and $x_0\in I_0$, then 
		\begin{equation*}
			\sum_{(k,q,x,p)\in A}T^{k}(q)\sigma(x)S^k(p) + \sigma(x_0)=0 
		\end{equation*}
		if and only if 
		\begin{equation*}
			\sum_{(k,q,x,p)\in A}T^{k}(q)T^n(\eta(x))S^k(p) + T^n(\eta(x_0))=0, 
		\end{equation*}
		\item \label{item:13} if $l\in\N$, $A$ is a finite subset of $\{(k,q,x,p)\mid k\in\N,\ q\in Q^{\otimes k},\ x\in I_0,\ p\in P^{\otimes l+k}\}$ and $B$ is a finite subset of $\{(x,p)\mid x\in I_0,\ p\in P^{\otimes l+k}\}$, then 
		\begin{equation*}
			\sum_{(k,q,x,p)\in A}T^{k}(q)\sigma(x)S^{l+k}(p) + \sum_{(q,x)\in B}\sigma(x)S^{l+k}(p)=0 
		\end{equation*}
		if and only if 
		\begin{equation*}
			\sum_{(k,q,x,p)\in A}T^{k}(q)T^n(\eta(x))S^{l+k}(p) + \sum_{(x,p)\in B}T^n(\eta(x))S^l(p)=0. 
		\end{equation*}
	\end{enumerate}
	We will just prove \eqref{item:4}. The other two claims can be proved in a similar way.
	
	To prove \eqref{item:4}, notice first that if $x\in I_0$ and $k\in\N$, then, since $I_0\subseteq J^{[\infty]}\subseteq J^{[k]}$, it follows from Lemma \ref{lemma:3} that there are $q_1,\dots,q_m\in Q^{\otimes k}$ and $p_1,\dots,p_m\in P^{\otimes k}$ such that $\sigma(x)=\sum_{i=1}^mT^k(q_i)S^k(p_i)$. It follows from condition (FS) that there are $q'_1,\dots,q'_r,q''_1,\dots,q''_s\in Q^{\otimes k}$ and $p'_1,\dots,p'_r,p''_1,\dots,p''_s\in P^{\otimes k}$ such that 
	\begin{equation*}
		\begin{split}
			\sigma(x)&=\sum_{i=1}^mT^k(q_i)S^k(p_i) =\sum_{j=1}^r\sum_{l=1}^s \sum_{i=1}^m T^k(q'_j)S^k(p'_j)T^k(q_i)S^k(p_i)T^k(q''_l)S^k(p''_l)\\
			&=\sum_{j=1}^r\sum_{l=1}^sT^k(q'_j)S^k(p'_j)\sigma(x)T^k(q''_l)S^k(p''_l) =\sum_{j=1}^r\sum_{l=1}^sT^k(q'_j)\sigma(\psi_k(p'_jx\otimes q''_l))S^k(p''_l). 
		\end{split}
	\end{equation*}
	Since $I_0$ is $\psi$-invariant, it follows that each $\psi_k(p'_jx\otimes q''_l)\in I_0$ and thus that 
	\begin{equation*}
		T^n(\eta(x))=\sum_{j=1}^r\sum_{l=1}^sT^k(q'_j)T^n(\eta(\psi_k(p'_jx\otimes q''_l)))S^k(p''_l). 
	\end{equation*}
	Thus it sufficies to show that if $k,l\in\N$ and $C$ is a finite subset of $\{(q,x,p)\mid q\in Q^{l+k},\ x\in I_0,\ p\in P^{\otimes k}\}$, then it is the case that $\sum_{(q,x,p)\in C}T^{l+k}(q)\sigma(x)S^k(p)=0$ if and only if $\sum_{(q,x,p)\in C}T^{l+k}(q)T^n(\eta(x))S^k(p)=0$, and that can be done using condition (FS) and the properties of $\eta$. 
\end{proof}

\section{Condition (L)} \label{sec:condition-L}

In this section condition (L) is introduced (Definition \ref{def:L}) and sufficient and necessary conditions for when every non-zero ideal in $\OCK$ contains a non-zero graded ideal (Theorem \ref{thm:L}) are given.
\begin{definition}
	\label{def:L} We say that a $\psi$-invariant ideal $I$ in $R$ is an \emph{$\psi$-invariant cycle} if there exist $n\in\N$ and an injective $R$-bimodule homomorphism $\eta:I\to Q^{\otimes n}$ such that $S_pT_{\eta(x)}(q)=\eta(\psi(px\otimes q))$ for $p\in P$, $x\in I$ and $q\in Q$, and we say that $J$ satisfies \emph{condition (L)} with respect to the $R$-system $(P,Q,\psi)$ if there are no non-zero $\psi$-invariant cycles $I$ in $R$ such that $I\subseteq J^{[\infty]}$.
	
	% both $I$ and
	% $\{\psi_n(p\otimes\eta(x))\mid p\in P^{\otimes n},\ x\in I\}$ are
	% contained in $J^{[\infty]}$ where $n$ and $\eta$ are as above.
\end{definition}

We will often, when it is clear from the context which $R$-system $(P,Q,\psi)$ we are working with, simply call a $\psi$-invariant cycle for an invariant cycle, and say that $J$ satisfies condition (L) instead of saying that it satisfies condition (L) with respect to $(P,Q,\psi)$.

Recall that if $(S',T',\sigma''B)$ is a covariant representation of $(P,Q,\psi)$, then $J_{(S',T',\sigma',B)}$ is defined to be the ideal $\{x\in R\mid \sigma'(x)\in\pi_{T',S'}(\mathcal{F}_P(Q)\}$ (see \cite[Definition 3.23]{CaOr2011}).
\begin{theorem}
	\label{thm:L} The following 4  conditions are equivalent: 
	\begin{enumerate}
		\item \label{item:14} The ideal $J$ satisfies condition (L). 
		\item \label{item:5} The subring $\FCK$ has the ideal intersection property. 
		\item \label{item:6} Every non-zero ideal in $\OCK$ contains a non-zero graded ideal. 
		\item \label{item:15} If $(S',T',\sigma''B)$ is an injective covariant representation of $(P,Q,\psi)$ and $J=J_{(S',T',\sigma',B)}$, then the ring homomorphism $\eta^J_{(S',T',\sigma',B)}:\OCK\to B$ from \cite[Theorem 3.29 (ii)]{CaOr2011} is injective. 
	\end{enumerate}
\end{theorem}
\begin{proof}
	$\eqref{item:14}\Leftrightarrow\eqref{item:5}$ follows from Proposition \ref{prop:main}.
	
	$\eqref{item:5}\Rightarrow\eqref{item:6}$: Let $K$ be a non-zero ideal in $\OCK$. Then $K\cap\FCK\ne \{0\}$ by assumption, and it follows from \cite[Lemma 3.35]{CaOr2011} that the ideal $H$ generated by $K\cap\FCK$ is graded. Since $H$ is obviously contained in $K$, this proves \eqref{item:6}.
	
	$\eqref{item:6}\Rightarrow\eqref{item:5}$: Let $K$ be a non-zero ideal in $\OCK$. By assumption there is a non-zero graded ideal $H$ such that $H\subseteq K$. It follows from \cite[Lemma 3.35]{CaOr2011} that $H\cap\FCK\ne\{0\}$, so also $K\cap\FCK\ne\{0\}$, which proves that $\FCK$ has the ideal intersection property.
	
	$\eqref{item:5}\Rightarrow\eqref{item:15}$: Let $H$ be the ideal in $\OCK$ generated by $\ker\eta^J_{(S',T',\sigma',B)}\cap\FCK$, and let $\wp:\OCK\to\OCK/H$ be the quotient map. Then $(\wp\circ S,\wp\circ T,\wp\circ\sigma,\OCK/H)$ is a surjective covariant representation of $(P,Q,\psi)$. It follows from \cite[Lemma 3.35]{CaOr2011} that $H$ is graded, from which it follows that the representation $(\wp\circ S,\wp\circ T,\wp\circ\sigma,\OCK/H)$ is graded (see \cite[Definition 3.20]{CaOr2011}). Since $H\subseteq \ker\eta^J_{(S',T',\sigma',B)}$, it follows that there is a ring homomorphism $\phi:\OCK/H\to B$ such that $\phi\circ\wp=\eta^J_{(S',T',\sigma',B)}$ and $\phi\circ\wp\circ S=S'$, $\phi\circ\wp\circ T=T'$ and $\phi\circ\wp\circ\sigma=\sigma'$. Since $(S',T',\sigma',B)$ is an injective representation, it follows that also $(\wp\circ S,\wp\circ T,\wp\circ\sigma,\OCK/H)$ is injective. It follows from \cite[Remark 3.13]{CaOr2011} that 
	\begin{equation*}
		J\subseteq J_{(\wp\circ S,\wp\circ T,\wp\circ\sigma,\OCK/H)} \subseteq J_{(S',T',\sigma',B)}=J. 
	\end{equation*}
	Thus $J_{(\wp\circ S,\wp\circ T,\wp\circ\sigma,\OCK/H)}=J$, and it follows from \cite[Theorem 3.29]{CaOr2011} that $\wp$ is an isomorphism, and thus that $\ker\eta^J_{(S',T',\sigma',B)}\cap\FCK=\{0\}$. It follows by assumption that $\ker\eta^J_{(S',T',\sigma',B)}=\{0\}$, and thus that $\eta^J_{(S',T',\sigma',B)}$ is injective.
	
	$\eqref{item:15}\Rightarrow\eqref{item:5}$: Let $K$ be an ideal in $\OCK$ such that $K\cap\FCK=\{0\}$, and let $\wp:\OCK\to\OCK/K$ be the quotient map. Then $(\wp\circ S,\wp\circ T,\wp\circ\sigma,\OCK/K)$ is a surjective covariant representation of $(P,Q,\psi)$. Since $\sigma(R)$ and $\pi_{T,S}(\mathcal{F}_P(Q))$ are subsets of $\FCK$ and $K\cap\FCK=\{0\}$, it follows from \cite[Proposition 3.28]{CaOr2011} that 
	\begin{equation*}
		J_{(\wp\circ S,\wp\circ T,\wp\circ\sigma,\OCK/H)} = J_{(S,T,\sigma,\OCK)}=J. 
	\end{equation*}
	Thus $\wp=\eta_{(\wp\circ S,\wp\circ T\wp\circ\sigma,\OCK/K)}$ is injective by assumption, and $K=\{0\}$ which proves that $\FCK$ has the ideal intersection property. 
\end{proof}

\section{The Cuntz-Krieger uniqueness theorem} \label{sec:cuntz-krieg-uniq}

In this Section the \emph{Cuntz-Krieger uniqueness property} is defined (Definition \ref{def:CK}), and the \emph{Cuntz-Krieger uniqueness result} is proved (Theorem \ref{thm:CK}).
\begin{definition}
	\label{def:CK} We say that the ideal $J$ has the \emph{Cuntz-Krieger uniqueness property} with respect to the $R$-system $(P,Q,\psi)$ if the following holds: 
	
	If $(S_1,T_1,\sigma_1,B_1)$ and $(S_2,T_2,\sigma_2,B_2)$ are two injective covariant representations of $(P,Q,\psi)$ and they are both Cuntz-Pimsner invariant relative to $J$, then there is a ring isomorphism $\phi$ between $\mathcal{R}\langle S_1,T_1,\sigma_1\rangle$ and $\mathcal{R}\langle S_2,T_2,\sigma_2\rangle$ such that $\phi\circ\sigma_1=\sigma_2$, $\phi\circ S_1=S_2$ and $\phi\circ T_1=T_2$. 
\end{definition}

We will often, when it is clear from the context which $R$-system $(P,Q,\psi)$ we are working with, simply say that $J$ has the Cuntz-Krieger uniqueness property instead of saying that it has the Cuntz-Krieger uniqueness property with respect to $(P,Q,\psi)$.

Recall from \cite[Definition 4.6]{CaOr2011} that $J$ is said to be a \emph{maximal} $\psi$-compatible ideal if $J=J'$ for any faithful $\psi$-compatible ideal $J'$ in $R$ satisfying $J\subseteq J'$.

\begin{theorem}
	\label{thm:CK} The following 5 conditions are equivalent: 
	\begin{enumerate}
		\item \label{item:9} The ideal $J$ has the Cuntz-Krieger uniqueness property. 
		\item \label{item:10} If $(S',T',\sigma',B)$ is an injective covariant representation of $(P,Q,\psi)$ which is Cuntz-Pimsner invariant relative to $J$, then the ring homomorphism $\eta_{(S',T',\sigma',B)}^J:\OCK \to B$ from \cite[Theorem 3.18]{CaOr2011} is injective. 
		\item \label{item:7} The subring $\sigma(R)$ has the ideal intersection property. 
		\item \label{item:8} The subring $\FCK$ has the ideal intersection property, and $J$ is a maximal $\psi$-compatible ideal. 
		\item \label{item:11} The ideal $J$ satisfies condition (L) and is a maximal $\psi$-compatible ideal. 
	\end{enumerate}
\end{theorem}

\begin{proof}
	$\eqref{item:9}\Rightarrow\eqref{item:10}$: The ring homomorphism $\eta_{(S',T',\sigma',B)}^J:\OCK \to B$ is the unique ring homomorphism from $\OCK$ to $B$ such that $\eta_{(S',T',\sigma',B)}^J\circ\sigma=\sigma'$, $\eta_{(S',T',\sigma',B)}^J\circ S=S'$ and $\eta_{(S',T',\sigma',B)}^J\circ T=T'$, so it follows by assumption that $\eta_{(S',T',\sigma',B)}^J$ is injective.
	
	$\eqref{item:10}\Rightarrow\eqref{item:9}$: If $(S_1,T_1,\sigma_1,B_1)$ and $(S_2,T_2,\sigma_2,B_2)$ are two injective covariant representations of $(P,Q,\psi)$ and there are both Cuntz-Pimsner invariant relative to $J$, then $\phi=\eta_{(S_2,T_2,\sigma_2,B_2)}^J\circ (\eta_{(S_1,T_1,\sigma_1,B_1)}^J)^{-1}$ is a ring isomorphism between $\mathcal{R}\langle S_1,T_1,\sigma_1\rangle$ and $\mathcal{R}\langle S_2,T_2,\sigma_2\rangle$ such that $\phi\circ\sigma_1=\sigma_2$, $\phi\circ S_1=S_2$ and $\phi\circ T_1=T_2$.
	
	$\eqref{item:10}\Rightarrow\eqref{item:7}$: Let $K$ be an ideal in $\OCK$ such that $K\cap\sigma(R)=\{0\}$, and let $\wp:\OCK\to\OCK/K$ be the quotient map. Then $(\wp\circ S,\wp\circ T,\wp\circ\sigma,\OCK/K)$ is an injective and surjective covariant representation of $(P,Q,\psi)$ which is Cuntz-Pimsner invariant relative to $J$. It follows by assumption that $\wp=\eta_{(\wp\circ S,\wp\circ T\wp\circ\sigma,\OCK/K)}$ is injective. Thus $K=\{0\}$, which proves that $\sigma(R)$ has the ideal intersection property.
	
	$\eqref{item:7}\Rightarrow\eqref{item:8}$: Since $\sigma(R)\subseteq\FCK$, it follows that $\FCK$ has the ideal intersection property if $\sigma(R)$ has. If $J$ is not a maximal $\psi$-invariant ideal, then there exists a $\psi$-compatible ideal $J'$ such that $J\subsetneq J'$. It follows from \cite[Remark 4.1]{CaOr2011} that $\rho_J(\mathcal{T}(J'))$ then would be a non-zero ideal in $\OCK$ with a zero intersection with $\sigma(R)$, which would mean that $\sigma(R)$ does not have the ideal intersection property. Thus it must be the case that $J$ is a maximal $\psi$-invariant ideal.
	
	$\eqref{item:8}\Rightarrow\eqref{item:10}$: Since $J$ is a maximal $\psi$-compatible ideal by assumption, it follows that $J_{(S',T',\sigma',B)}=J$. Thus it follows from Theorem \ref{thm:L} that $\eta^J_{(S',T',\sigma',B)}$ is injective.
	
	$\eqref{item:8}\Leftrightarrow\eqref{item:11}$ follows from Theorem \ref{thm:L}. 
\end{proof}

\section{Simplicity of $\OCK$} \label{sec:simplicity-ock}

In this section sufficient and necessary conditions for when $\OCK$ is simple are given (Theorem \ref{thm:simple}).
\begin{definition}
	We say that $J$ is a \emph{super maximal} $\psi$-compatible ideal if the only $T$-pairs $(I,J')$ of $(P.Q,\psi)$ which satisfies that $J\subseteq J'$, are $(0,J)$ and $(R,R)$. 
\end{definition}

Since $(0,J')$ is a $T$-pair of $(P.Q,\psi)$ for any any faithful $\psi$-compatible ideal $J'$ in $R$, it follows that if $J$ is a super maximal $\psi$-compatible ideal, then it is also a maximal $\psi$-compatible ideal.

\begin{remark}
	\label{remark:supermax} It follows from \cite[Theorem 7.27]{CaOr2011} that $J$ is a  super maximal $\psi$-compatible ideal if and only if the only graded ideals in $\OCK$ are $\{0\}$ and $\OCK$.
\end{remark}

\begin{theorem}
	\label{thm:simple} The following 5 conditions are equivalent: 
	\begin{enumerate}
		\item \label{item:16} The ring $\OCK$ is simple. 
		\item \label{item:17} The subring $\sigma(R)$ has the ideal intersection property and $J$ is a super maximal $\psi$-compatible ideal. 
		\item \label{item:18} The subring $\FCK$ has the ideal intersection property and $J$ is a super maximal $\psi$-compatible ideal. 
		\item \label{item:19} The ideal $J$ satisfies condition (L) and is a super maximal $\psi$-compatible ideal.
		\item \label{item:20} If $(S',T',\sigma',B)$ is a non-zero covariant representation of $(P,Q,\psi)$ which is Cuntz-Pimsner invariant relative to $J$, then the ring homomorphism 
		$$\eta_{(S',T',\sigma',B)}^J:\OCK \to B$$ 
		from \cite[Theorem 3.18]{CaOr2011} is injective. 
	\end{enumerate}
\end{theorem}

\begin{proof}
	$\eqref{item:16}\Rightarrow\eqref{item:17}$: If $\OCK$ is simple, then clearly $\sigma(R)$ has the ideal intersection property. If $(I,J')$ is a $T$-pair of $(P,Q,\psi)$ different from $(0,J)$, then it follows from \cite[Theorem 7.27]{CaOr2011} that $H^J_{(I,J')}$ is a non-zero ideal in $\OCK$. If $\OCK$ is simple, then that would imply that $H^J_{(I,J')}=\OCK$ and thus $(I,J')=(R,R)$ from which it follows that $J$ is a super maximal $\psi$-compatible ideal.
	
	$\eqref{item:17}\Leftrightarrow\eqref{item:18}$ and $\eqref{item:18}\Leftrightarrow\eqref{item:19}$ follow from Theorem \ref{thm:CK} and the fact that $J$ is a maximal $\psi$-compatible ideal if it is a super maximal $\psi$-compatible ideal.
	
	$\eqref{item:17}\Rightarrow\eqref{item:20}$: It follows from \cite[Proposition 7.8]{CaOr2011} that $(I_{(S',T',\sigma',B)},J_{(S',T',\sigma',B)})$ is a $T$-pair. Since $(S',T',\sigma',B)$ is Cuntz-Pimsner invariant relative to $J$, it follows from \cite[Remark 3.25]{CaOr2011} that $J\subseteq J_{(S',T',\sigma',B)}$, and since $(S',T',\sigma',B)$ is non-zero, it follows from \cite[Theorem 7.11]{CaOr2011} that $(I_{(S',T',\sigma',B)},J_{(S',T',\sigma',B)})\ne (R,R)$. Thus $(I_{(S',T',\sigma',B)},J_{(S',T',\sigma',B)})=(0,J)$ which implies that $(S',T',\sigma',B)$ is an injective representation. It then follows from Theorem \ref{thm:CK} that $\eta_{(S',T',\sigma',B)}^J$ is injective.
	
	$\eqref{item:20}\Rightarrow\eqref{item:16}$: Let $K$ be a proper ideal in $\OCK$, and let $\wp:\OCK\to\OCK/K$ be the quotient map. Then $(\wp\circ S,\wp\circ T,\wp\circ\sigma,\OCK/K)$ is a surjective covariant representation of $(P,Q,\psi)$ which is Cuntz-Pimsner invariant relative to $J$. It follows by assumption that $\wp=\eta_{(\wp\circ S,\wp\circ T\wp\circ\sigma,\OCK/K)}$ is injective. Thus $K=\{0\}$ which proves that $\OCK$ is simple. 
\end{proof}

\section{Condition (K)} \label{sec:condition-k}

In this section condition (K) is introduced (Definition \ref{def:K}), and sufficient and necessary conditions for when every ideal in $\OCK$ is graded are given (Theorem \ref{thm:K}).

Recall from \cite[Section 7]{CaOr2011} that if $I$ is a $\psi$-invariant ideal in $R$, then $R_I=R/I$, $Q_I=Q/QI$ and ${}_IP=P/IP$, and $\wp_I$ denote the corresponding quotient map. Recall also that there is an $R_I$-bimodule homomorphism $\psi_I:{}_IP\otimes Q_I\to R_I$ given by $\psi_I(\wp_I(p)\otimes\wp_I(q))=\wp_I(\psi(p\otimes q))$. The triple $({}_IP,Q_I,\psi_I)$ is then an $R_I$-system satisfying condition (FS) (see \cite[Lemma 7.4]{CaOr2011}). When $(I,J')$ is a $T$-pair, then $J'_I$ denote the faithful $\psi_I$-compatible ideal $\wp_I(J')$ in $R_I$.
\begin{definition}
	\label{def:K} We say that the ideal $J$ satisfies \emph{condition (K)} with respect to the $R$-system $(P,Q,\psi)$ if $J'_I$ satisfies condition (L) with respect to the $R_I$-system $({}_IP,Q_I,\psi_I)$ whenever $(I,J')$ is a $T$-pair of $(P,Q,\psi)$ such that $J\subseteq J'$. 
\end{definition}

We will often, when it is clear from the context which $R$-system $(P,Q,\psi)$ we are working with, simply say that $J$ satisfies condition (K) instead of saying that it satisfies condition (K) with respect to $(P,Q,\psi)$.
\begin{theorem}
	\label{thm:K} The following 3 conditions are equivalent: 
	\begin{enumerate}
		\item \label{item:21} Every ideal of $\OCK$ is graded. 
		\item \label{item:22} The ideal $J$ satisfies condition (K). 
		\item \label{item:23} If $(S',T',\sigma',B)$ is a covariant representation of $(P,Q,\psi)$ which is Cuntz-Pimsner invariant relative to $J$, and $(I,J')=\omega_{(S',T',\sigma',B)}$, then the ring homomorphism $\eta^{(I,J')}_{(S',T',\sigma',B)}:\mathcal{O}_{({}_IP,Q_I,\psi_I)}(J'_I)\to B$ from \cite[Theorem 7.11 (ii)]{CaOr2011} is injective. 
	\end{enumerate}
\end{theorem}
\begin{proof}
	$\eqref{item:21}\Rightarrow\eqref{item:22}$: Let $\omega=(I,J')$ be a $T$-pair of $(P,Q,\psi)$ such that $J\subseteq J'$ and let $H$ be a non-zero ideal in $\mathcal{O}_{({}_IP,Q_I,\psi_I)}(J'_I)$. Recall from \cite[Page 36]{CaOr2011} that there is a covariant representation $(\iota^\omega_P, \iota^\omega_Q, \iota^\omega_R,\mathcal{O}_{({}_IP,Q_I,\psi_I)}(J'_I))$ such that $\iota^\omega_P=\iota^{J'_I}_{{}_IP}\circ\wp_I$, $\iota^\omega_Q=\iota^{J'_I}_{Q_I}\circ\wp_I$ and $\iota^\omega_R=\iota^{J'_I}_{R_I}\circ\wp_I$. It follows from \cite[Remark 3.25 and Theorem 3.29]{CaOr2011} that there is a surjective graded ring homomorphism $\phi:\OCK\to \mathcal{O}_{({}_IP,Q_I,\psi_I)}(J'_I)$ which intertwines the two representations $(S,T,\sigma,\OCK)$ and $(\iota^\omega_P, \iota^\omega_Q, \iota^\omega_R,\mathcal{O}_{({}_IP,Q_I,\psi_I)}(J'_I))$. We then have that $\phi^{-1}(H)$ is an ideal in $\OCK$. Thus $\phi^{-1}(H)$ is graded by assumption. It follows that also $H$ is graded. It therefore follows from Theorem \ref{thm:L} that $J'_I$ satisfies condition (L) with respect to the $R_I$-system $({}_IP,Q_I,\psi_I)$. This proves that $J$ satisfies condition (K).
	
	$\eqref{item:22}\Rightarrow\eqref{item:23}$: It follows from \cite[Lemma 7.10]{CaOr2011} that there is an injective covariant representation $(S_I,T_I,\sigma_I,B)$ of $({}_IP,Q_I,\psi_I)$ such that $S_I=S'\circ\wp_I$, $T_I=T'\circ\wp_I$ and $\sigma_I=\sigma'\circ\wp_I$. Since $\pi_{T_I,S_I}(\mathcal{F}_{{}_IP}(Q_I))=\pi_{T',S'}(\mathcal{F}_P(Q))$, it follows that $J_{(S_I,T_I,\sigma_I,B)}=\wp_I(J_{(S',T',\sigma',B)})=\wp_I(J')=J'_I$. It therefore follows from Theorem \ref{thm:L} that $\eta^{(I,J')}_{(S',T',\sigma',B)}=\eta^{J'}_{(S_I,T_I,\sigma_I,B)}$ is injective.
	
	$\eqref{item:23}\Rightarrow\eqref{item:21}$: Let $H$ be an ideal in $\OCK$ and let $\wp:\OCK\to\OCK/H$ be the quotient map. Then $(\wp\circ S,\wp\circ T,\wp\circ\sigma,\OCK/H)$ is a covariant representation which is Cuntz-Pimsner invariant relative to $J$. Let $(I,J')=\omega_{(\wp\circ S,\wp\circ T,\wp\circ\sigma,\OCK/H)}$. Then $\eta^{(I,J')}_{(\wp\circ S,\wp\circ T,\wp\circ\sigma,\OCK/H)}$ is injective by assumption. Since $\oplus_{n\in\Z}\mathcal{O}^{(n)}_{({}_IP,Q_I,\psi_I)}(J'_I)$ is a $\Z$-grading of $\mathcal{O}_{({}_IP,Q_I,\psi_I)}(J'_I)$, it follows that 
	\begin{equation*}
		\oplus_{n\in\Z}\wp(\mathcal{O}^{(n)}_{(P,Q,\psi)}(J)) =\oplus_{n\in\Z} \eta^{(I,J')}_{(\wp\circ S,\wp\circ T,\wp\circ\sigma,\OCK/H)} (\mathcal{O}^{(n)}_{({}_IP,Q_I,\psi_I)}(J'_I)) 
	\end{equation*}
	is a $\Z$-grading of $\OCK/H$. Thus $H$ is graded. 
\end{proof}
\begin{remark}
	\label{remark} It follows from the above theorem that if $J$ satisfies condition (K), then \cite[Theorem 7.27]{CaOr2011} gives a bijective correspondence between the set of all ideals of $\OCK$ and the set of $T$-pairs $(I,J')$ of $(P,Q,\psi)$ satisfying $J\subseteq J'$. 
\end{remark}

\section{Toeplitz rings} \label{sec:toeplitz-rings}

When $J=\{0\}$, then $\OCK$ is the Toeplitz ring $\toeplitz$ and $J$ automatically satisfies condition (L). Thus the following 3 corollaries follow from Theorem \ref{thm:L}, Theorem \ref{thm:CK} and Theorem \ref{thm:simple}, respectively.

\begin{cor}
	If $(S',T',\sigma',B)$ is an injective covariant representation of $(P,Q,\psi)$, then the ring homomorphism $\eta_{(S',T',\sigma',B)}:\toeplitz\to B$ from \cite[Theorem 1.7]{CaOr2011} is injective if and only if $J_{(S',T',\sigma',B)}=\{0\}$. 
\end{cor}

\begin{cor}\label{cor:Tuniq}
	Assume that there are no non-zero faithful $\psi$-compatible ideals of $R$. If $(S_1,T_1,\sigma_1,B_1)$ and $(S_2,T_2,\sigma_2,B_2)$ are two injective covariant representations of $(P,Q,\psi)$, then there is a ring isomorphism $\phi$ between $\mathcal{R}\langle S_1,T_1,\sigma_1\rangle$ and $\mathcal{R}\langle S_2,T_2,\sigma_2\rangle$ such that $\phi\circ\sigma_1=\sigma_2$, $\phi\circ S_1=S_2$ and $\phi\circ T_1=T_2$. 
\end{cor}

\begin{cor}
	The Toeplitz ring $\toeplitz$ is simple if and only if $(0,0)$ and $(R,R)$ are the only $T$-pairs of $(P,Q,\psi)$. 
\end{cor}

\section{Leavitt path algebras} \label{sec:leavitt} 
We will in this section show how we can recover from the results obtained in this paper Theorem 6.8, Corollary 6.10, Theorem 6.16, Corollary 6.17 and Theorem 6.18 of \cite{To2007} and obtain an algebraic analogue of \cite[Theorem 4.1]{FoRa1999}.

Let $(E^0,E^1,r,s)$ be a directed graph (ie. $E^0$ and $E^1$ are sets and $r$ and $s$ are maps from $E^1$ to $E^0$) and let $F$ be a field. When $n$ is a positive integer, then we let $E^n$ be the set $\{(e_1,e_2,\dots,e_n)\in E^1\times E^1\times\dots\times E^1\mid r(e_i)=s(e_{i+1})\text{ for }i=1,2,\dots,n-1\}$. For $\alpha=(e_1,e_2,\dots,e_n)\in E^n$ we define $s(\alpha)$ to be $s(e_1)$ and $r(\alpha)$ to be $r(e_n)$. For each $v\in E^0$ we let $vE^n$ denote the set $\{\alpha\in E^n\mid s(\alpha)=v\}$ and we let $E^nv$ denote the set $\{\alpha\in E^n\mid r(\alpha)=v\}$. A \emph{closed path} is an $\alpha\in E^n$ such that $r(\alpha)=s(\alpha)$. The element $s(\alpha)$ is called the \emph{base} of $\alpha$. A closed path $\alpha=(e_1,e_2,\dots,e_n)$ is said to be \emph{simple} if $s(e_i)\ne s(e_1)$ for each $i=2,3,\dots,n$, and to have an \emph{exit} if $|s(e_i)E^1|>1$ for some $i\in\{1,2,\dots,n\}$.

Following \cite[Example 5.8]{CaOr2011} we define $R$ be the ring $\oplus_{v\in E^0}R_v$ where each $R_v$ is a copy of $F$; we let $Q$ be $R$-bimodule $\oplus_{e\in E^1}Q_e$ where each $Q_e$ is a copy of $F$ and the left and the right multiplication are defined by 
\begin{align*}
	\left(\sum_{e\in E^1} q_e\textbf{1}_{e}\right)\cdot \left(\sum_{v\in E^0} r_v\textbf{1}_{v}\right) &= \sum_{e\in E^1} q_e r_{r(e)}\textbf{1}_{e} \\
	\left(\sum_{v\in E^0} r_v\textbf{1}_{v}\right)\cdot\left(\sum_{e\in E^1} q_e\textbf{1}_{e}\right) &= \sum_{e\in E^1} r_{s(e)} q_e\textbf{1}_{e} 
\end{align*}
where $\textbf{1}_v$ denotes the unit of $R_v$, $\textbf{1}_e$ denotes the unit of $Q_e$, and $\{r_v\}_{v\in E^0}$ and $\{q_e\}_{e\in E^1}$ are families of elements from $F$ with only a finite number of non-zero elements; we let $P$ be the $R$-bimodule $\oplus_{e\in E^1} P_e$ where each $P_e$ is a copy of $F$ and the left and the right multiplication are defined by 
\begin{align*}
	\left(\sum_{e\in E^1} p_e\textbf{1}_{\overline{e}}\right)\cdot \left(\sum_{v\in E^0} r_v\textbf{1}_{v}\right) &= \sum_{e\in E^1} p_e r_{s(e)}\textbf{1}_{\overline{e}} \\
	\left(\sum_{v\in E^0} r_v\textbf{1}_{v}\right)\cdot\left(\sum_{e\in E^1} p_e\textbf{1}_{\overline{e}}\right) &= \sum_{e\in E^1} r_{r(e)} p_e\textbf{1}_{\overline{e}} 
\end{align*}
where $\textbf{1}_{\overline{e}}$ denotes the unit of $P_e$, and $\{r_v\}_{v\in E^0}$ and $\{p_e\}_{e\in E^1}$ are families of elements from $F$ with only a finite number of non-zero elements; and we define $\psi:P\otimes_R Q \to R$ to be the $R$-bimodule homomorphism given by 
\begin{equation*}
	\left(\sum_{e\in E^1} p_e\textbf{1}_{\overline{e}}\right)\otimes \left(\sum_{e\in E^1} q_e\textbf{1}_{e}\right)\mapsto \sum_{v\in E^0} \left(\sum_{e\in E^1v} p_{e} q_e \right)\textbf{1}_{v}, 
\end{equation*}
then $(P,Q,\psi)$ is an $R$-system. Recall also that if we let $J$ be the ideal $\text{span}_F\{\textbf{1}_v\mid v\in E^0,\ 0<|vE^1|<\infty \}\subseteq R$, then $J$ is a maximal faithful $\psi$-compatible ideal and $\OCK$ is isomorphic to the Leavitt path algebra of $(E^0,E^1)$ (see for example \cite{AbAr2005,AbAr2008} and \cite{To2007}). It is straightforward to check that $J^{[n]}=\text{span}_F\{\textbf{1}_v\mid v\in E^0,\ 0<|vE^n|<\infty \}$ for each $n\in\N$ from which it follows that $J^{[\infty]}=\text{span}_F\{\textbf{1}_v\mid v\in E^0,\ 0<|vE^n|<\infty \text{ for all }n\in\N\}$.

Suppose that $I$ is a non-zero $\psi$-invariant cycle and let $\eta:I\to Q^{\otimes n}$ be an injective $R$-bimodule homomorphism satisfying $S_pT_{\eta(x)}(q)=\eta(\psi(px\otimes q))$ for $p\in P$, $x\in I$ and $q\in Q$. We will prove that it follows that $(E^0,E^1,r,s)$ has a closed path without an exit. We can, and will, identify $Q^{\otimes n}$ with the $R$-bimodule $\oplus_{\alpha\in E^n}Q_\alpha$ where each $Q_\alpha$ is a copy of $F$ and the left and the right multiplication are defined by
\begin{align*}
	\left(\sum_{\alpha\in E^n} q_\alpha\textbf{1}_{\alpha}\right)\cdot \left(\sum_{v\in E^0} r_v\textbf{1}_{v}\right) &= \sum_{\alpha\in E^n} q_\alpha r_{r(\alpha)}\textbf{1}_{\alpha} \\
	\left(\sum_{v\in E^0} r_v\textbf{1}_{v}\right)\cdot\left(\sum_{\alpha\in E^n} q_\alpha\textbf{1}_{\alpha}\right) &= \sum_{\alpha\in E^n} r_{s(\alpha)} q_\alpha\textbf{1}_{\alpha} 
\end{align*}
where $\textbf{1}_\alpha$ denote the unit of $Q_\alpha$, and $\{r_v\}_{v\in E^0}$ and $\{q_\alpha\}_{\alpha\in E^n}$ are families of elements of $F$ with only a finite number of non-zero elements. Likewise, we identify $P^{\otimes n}$ with the $R$-bimodule $\oplus_{\alpha\in E^n}P_\alpha$ where each $P_\alpha$ is a copy of $F$ and the left and the right multiplication are defined by 
\begin{align*}
	\left(\sum_{\alpha\in E^n} p_\alpha\textbf{1}_{\overline{\alpha}}\right)\cdot \left(\sum_{v\in E^0} r_v\textbf{1}_{v}\right) &= \sum_{\alpha\in E^n} p_\alpha r_{r(\alpha)}\textbf{1}_{\overline{\alpha}} \\
	\left(\sum_{v\in E^0} r_v\textbf{1}_{v}\right)\cdot\left(\sum_{\alpha\in E^n} p_\alpha\textbf{1}_{\overline{\alpha}}\right) &= \sum_{\alpha\in E^n} r_{s(\alpha)} p_\alpha\textbf{1}_{\overline{\alpha}} 
\end{align*}
where $\textbf{1}_{\overline{\alpha}}$ denote the unit of $P_\alpha$, and $\{r_v\}_{v\in E^0}$ and $\{p_\alpha\}_{\alpha\in E^n}$ are families of elements of $F$ with only a finite number of non-zero elements. We then have that $\psi_n:P^{\otimes n}\otimes Q^{\otimes n}\to R$ is given by 
\begin{equation*}
	\left(\sum_{\alpha\in E^n} p_\alpha\textbf{1}_{\overline{\alpha}}\right)\otimes \left(\sum_{\alpha\in E^n} q_\alpha\textbf{1}_{\alpha}\right)\mapsto \sum_{v\in E^0} \left(\sum_{\alpha\in E^nv} p_{\alpha} q_\alpha \right)\textbf{1}_{v}. 
\end{equation*}

Let $H$ be the set $\{v\in E^0\mid \textbf{1}_v\in I\}$. It follows from the $\psi$-invariance of $I$ that $H$ is hereditary (that is, whenever $e\in E^1$ with $s(e)\in H$, then $r(e)\in H$). Let $v\in H$. Then $\eta(\textbf{1}_v)=\sum_{\alpha\in K}f_\alpha\textbf{1}_\alpha$ for some non-empty finite subset $K\subseteq E^n$ and non-zero elements $f_\alpha\in F,\ \alpha\in K$. Since $\textbf{1}_v\eta(\textbf{1}_v)\textbf{1}_v= \eta(\textbf{1}_v\textbf{1}_v\textbf{1}_v)=\eta(\textbf{1}_v)$, it follows that $r(\alpha)=s(\alpha)=v$ for each $\alpha\in K$. Let $\alpha\in E^n$ with $r(\alpha)=s(\alpha)=v$. Since 
\begin{equation*}
	\psi_n\bigl(\textbf{1}_{\overline{\alpha}}\otimes\eta(\textbf{1}_v)\bigr)\textbf{1}_\alpha= \eta\bigl(\psi_n(\textbf{1}_{\overline{\alpha}}\textbf{1}_v\otimes\textbf{1}_\alpha)\bigr)= \eta(\textbf{1}_v), 
\end{equation*}
it follows that $K\subseteq\{\alpha\}$. Hence it must be the case that there is exactly one $\alpha_v\in E^n$ with $r(\alpha)=s(\alpha)=v$, and that $K$ consists of this element. Thus there is for each $v\in H$ a unique $\alpha_v\in E^n$ with $r(\alpha)=s(\alpha)=v$ and $\eta(\textbf{1}_v)=f_{\alpha_v}\textbf{1}_{\alpha_v}$ for some $f_{\alpha_v}\in F\setminus\{0\}$.

Let $v\in H$, let $\alpha_v=(e_1,e_2,\dots,e_n)$ and assume that there is an $e'\in E^1\setminus\{e_1\}$ with $s(e)=v$. Then 
\begin{equation*}
	\eta(\textbf{1}_{r(e)})=\eta\bigl(\psi(\textbf{1}_{\overline{e}}\textbf{1}_v\otimes\textbf{1}_e)\bigr)= S_{\textbf{1}_{\overline{e}}}T_{\eta(\textbf{1}_v)}\textbf{1}_e= f_{\alpha_v}S_{\textbf{1}_{\overline{e}}}T_{\textbf{1}_{\alpha_v}}\textbf{1}_e=0 
\end{equation*}
which contradicts the fact that $\eta$ is injective. Thus, for each $v\in H$ it is the case that $vE^1=\{e_1\}$ where $e_1$ is the initial part of $\alpha_v$. It follows that every $v\in H$ is the base of a closed path which has no exit. In particular, $(E^0,E^1,r,s)$ has a closed path which has no exit.

On the other hand, it is straightforward to check that if $\alpha_v=(e_1,e_2,\dots,e_n)$ is a closed path without an exit, then $H=\{s(e_i)\mid i\in\{1,2,\dots,n\}$ is a hereditary subset of $E^0$, $I=\text{span}_F\{\textbf{1}_v\mid v\in H\}$ is contained in $J^{[\infty]}$ and is a $\psi$-invariant ideal in $R$, and the $F$-linear map $\eta:I\to Q^{\otimes n}$ given by $\textbf{1}_{s(e_i)}\mapsto\textbf{1}_{(e_i,e_{i+1},\dots,e_n,e_1,e_2,\dots,e_{i-1})}$ for $i\in\{1,2,\dots,n\}$ is an injective $R$-bimodule homomorphism $\eta:I\to Q^{\otimes n}$ satisfying $S_pT_{\eta(x)}(q)=\eta(\psi(px\otimes q))$ for $p\in P$, $x\in I$ and $q\in Q$. Thus $J$ satisfies condition (L) if and only every closed path in $(E^0,E^1,r,s)$ has an exit (cf. \cite[Definition 6.3]{To2007}). We therefore recover \cite[Theorem 6.8 and Corollary 6.10]{To2007} from Theorem \ref{thm:CK}. By combining \cite[Theorem 5.7 and Proposition 6.12]{To2007} and \cite[Example 7.31]{CaOr2011} with the above characterization of when $J$ satisfies condition (L), one sees that $J$ satisfies condition (K) if and only if every $v\in E^0$ is either the base of no closed path or the base of at least two simple closed paths (cf. \cite[Definition 6.11]{To2007}). We therefore recover \cite[Theorem 6.16 and Corollary 6.17]{To2007} from Theorem \ref{thm:K} and Remark \ref{remark}. Finally, it follows from \cite[Theorem 5.7]{To2007} (cf. \cite[Example 7.31]{CaOr2011}) and Remark \ref{remark:supermax} that $J$ is super maximal if and only if the only saturated hereditary subsets of $E^0$ are $\emptyset$ and $E^0$, thus we recover \cite[Theorem 6.18]{To2007} from Theorem \ref{thm:simple} and the above characterization of when $J$ satisfies condition (L).

We will end this subsection by using Corollary \ref{cor:Tuniq} to give a uniqueness theorem for the Toeplitz ring $\mathcal{T}_{(P,Q,\psi)}=\mathcal{O}_{(P,Q,\psi)}(0)$. 

\begin{definition}\label{def:TCK}
	Let $E=(E^0,E^1,r,s)$ be a directed graph, let $F$ be a field and $B$ an $F$-algebra. A \emph{Toeplitz-Cuntz-Krieger $E$-family} in $B$ consists of a family $\{p_v\mid v\in E^0\}$ of pairwise orthogonal idempotents in $B$ together with a family $\{x_e,y_e\mid e\in E^1\}$ of elements in $B$ satisfying the following relations
	\begin{enumerate}
	  \item $p_{s(e)}x_{e}=x_{e}=x_{e}p_{r(e)}$ for
	    $e\in E^1$,
	  \item $p_{r(e)}y_{e}=y_{e}=y_{e}p_{s(e)}$
	    for $e\in E^1$,
	  \item $y_{e}x_{f}=\delta_{e,f}p_{r(e)}$ for $e,f\in E^1$,
	  \end{enumerate}
	where $\delta_{e,f}$ denotes the Kronecker's delta function.
\end{definition}

\begin{theorem}\label{thm:TCK}
	Let $E=(E^0,E^1,r,s)$ be a directed graph and let $F$ be a field. Let $R$ and $(P,Q,\psi)$ be as defined above and let $(S,T,\sigma,\mathcal{T}_{(P,Q,\psi)})$ be the Toeplitz representation of $(P,Q,\psi)$. Then $\{\sigma(\textbf{1}_v)\mid v\in E^0\}$ together with $\{T(\textbf{1}_e), S(\textbf{1}_{\overline{e}})\mid e\in E^1\}$ is a Toeplitz-Cuntz-Krieger $E$-family. If $B$ is an $F$-algebra and $\{p_v\mid v\in E^0\}$ together with $\{x_e,y_e\mid e\in E^1\}$ is a Toeplitz-Cuntz-Krieger $E$-family, then there exists a unique $F$-algebra homomorphism $\eta:\mathcal{T}_{(P,Q,\psi)}\to B$ satisfying $\eta(\sigma(\textbf{1}_v))=p_v$ for $v\in E^0$, and $\eta(T(\textbf{1}_e))=x_e$ and $\eta(S(\textbf{1}_{\overline{e}}))=y_e$ for $e\in E^1$. The homomorphism $\eta$ is injective if and only if $p_v\ne 0$ for each $v\in E^0$ and $p_v\ne\sum_{e\in vE^1}x_ey_e$ for $v\in E^0$ with $0<\abs{vE^1}<\infty$.
\end{theorem}

\begin{proof}
	That $\mathcal{T}_{(P,Q,\psi)}$ is an $F$-algebra and that $\{\sigma(\textbf{1}_v)\mid v\in E^0\}\cup \{T(\textbf{1}_e), S(\textbf{1}_{\overline{e}})\mid e\in E^1\}$ is a Toeplitz-Cuntz-Krieger $E$-family is proved in \cite[Example 1.10]{CaOr2011}. It is also proved in \cite[Example 1.10]{CaOr2011} that if $B$ is an $F$-algebra and $\{p_v\mid v\in E^0\}$ together with  $\{x_e,y_e\mid e\in E^1\}$ is a Toeplitz-Cuntz-Krieger $E$-family, then there is a covariant representation $(S',T',\sigma',B)$ of $(P,Q,\psi)$ such that $S'(\lambda\textbf{1}_{\overline{e}})=\lambda y_e$ and $T'(\lambda \textbf{1}_e)=\lambda x_e$ for $e\in E^1$ and $\lambda\in F$, and $\sigma'(\lambda\textbf{1}_v)=\lambda p_v$ for $v\in E^0$ and $\lambda\in F$. It then follows from \cite[Theorem 1.7]{CaOr2011} that there is a ring homomorphism $\eta:\mathcal{T}_{(P,Q,\psi)}\to B$ such that $\eta(\sigma(\lambda\textbf{1}_v))=\sigma'(\lambda\textbf{1}_v)=\lambda p_v$ for $v\in E^0$ and $\lambda\in F$, and $\eta(T(\lambda\textbf{1}_e))=T'(\lambda\textbf{1}_e)=\lambda x_e$ and $\eta(S(\lambda\textbf{1}_{\overline{e}}))=S'(\lambda\textbf{1}_{\overline{e}})=\lambda y_e$ for $e\in E^1$ and $\lambda\in F$. It follows that $\eta$ is a $F$-algebra homomorphism and that $\eta(\sigma(\textbf{1}_v))=p_v$ for $v\in E^0$, and $\eta(T(\textbf{1}_e))=x_e$ and $\eta(S(\textbf{1}_{\overline{e}}))=y_e$ for $e\in E^1$. Since $\mathcal{T}_{(P,Q,\psi)}$ is generated, as an $F$-algebra, by $\{\sigma(\textbf{1}_v)\mid v\in E^0\}\cup\{T(\textbf{1}_e), S(\textbf{1}_{\overline{e}})\mid e\in E^1\}$, there cannot be any other $F$-algebra homomorphism from $\mathcal{T}_{(P,Q,\psi)}$ to $B$ which for every $v\in E^0$ maps $\sigma(\textbf{1}_v)$ to $p_v$ and for any $e\in E^1$ maps $T(\textbf{1}_e)$ to $x_e$ and $S(\textbf{1}_{\overline{e}})$ to $y_e$.
	
	The map $\sigma$ is injective by \cite[Theorem 1.7]{CaOr2011}. It follows that if $\eta$ is injective, then $p_v\ne 0$ for each $v\in E^0$. Assume that $p_v\ne 0$ for each $v\in E^0$. Since $R=\oplus_{v\in E^0}R_v$ where each $R_v$ is a copy of $F$, it follows that $\sigma'$ is injective. Thus it follows from Corollary \ref{cor:Tuniq} that $\eta$ is injective if and only if $J_{(S',T',\sigma',B)}=0$. It follows from \cite[Lemma 3.24]{CaOr2011} that 
	\begin{equation*}
		J_{(S',T',\sigma',B)}=\bigl\{r\in\Delta^{-1}(\mathcal{F}_P(Q))\mid \sigma'(r)=\pi_{T',S'}(\Delta(r))\bigr\}.
	\end{equation*}
	It is proved in \cite[Example 5.8]{CaOr2011} that
	\begin{equation*}
		\Delta^{-1}(\mathcal{F}_P(Q))=\spn_F\{\textbf{1}_v\mid 0<\abs{vE^1}<\infty\},
	\end{equation*}
	and is straightforward to check that $\Delta(\textbf{1}_v)=\sum_{e\in vE^1}\theta_{\textbf{1}_e,\textbf{1}\overline{e}}$ if $\textbf{1}_v\in \Delta^{-1}(\mathcal{F}_P(Q))$. It follows that 
	\begin{equation*}
		J_{(S',T',\sigma',B)}=\spn_F\Bigl\{\textbf{1}_v\bigm| 0<\abs{vE^1}<\infty,\ p_v=\sum_{e\in vE^1}x_ey_e\Bigr\}.
	\end{equation*}
		Thus $\eta$ is injective if and only if $p_v\ne 0$ for each $v\in E^0$ and $p_v\ne\sum_{e\in vE^1}x_ey_e$ for $v\in E^0$ with $0<\abs{vE^1}<\infty$.
\end{proof}

Theorem \ref{thm:TCK} is the algebraic analogue of \cite[Theorem 4.1]{FoRa1999}.

\section{Crossed products of a ring by an automorphism and fractional skew monoid rings of a corner isomorphism} \label{sec:cross}

We will in this section use Theorem \ref{thm:simple} to give a characterization of when the fractional skew monoid ring of a ring isomorphism is simple (Corollary \ref{cor:fractional}), and when the crossed product of a ring by an automorphism is simple (Corollary \ref{cor:auto}).

A ring $R$ has \emph{local units} if given any finite set $F\subseteq R$ there exists an idempotent $e\in R$ such that $er=re=r$ for every $r\in F$, in other words, the set of all idempotents of $R$, $\text{Idem}(R)$, is a directed system (with order $e\leq f$ if and only if $ef=fe=e$) and $R=\bigcup_{e\in \text{Idem}(R)} eRe$.

Let $R$ be a ring with local units and let $\alpha:R\to R$ be an injective ring homomorphism such that $\alpha(R)R\alpha(R)\subseteq\alpha(R)$ (notice this is equivalent to $\alpha(R)R\alpha(R)=\alpha(R)$ since $R$ has local units). Recall from \cite[Example 5.6]{CaOr2011} that if $P$ is the $R$-bimodule which is equal to $\spn\{r_1\alpha(r_2)\mid r_1,r_2,\in R\}$ as a set, has the additive structure it inherits from $R$, and has the left and right actions given by $r\cdot p=rp$ and $p\cdot r=p\alpha(r)$ for $r\in R$ and $p\in P$; $Q$ is the $R$-bimodule which is equal to $\spn\{\alpha(r_1)r_2\mid r_1,r_2\in R\}$ as a set, has the additive structure it inherits from $R$, and has the left and right given by $r\cdot q=\alpha(r)q$ and $q\cdot r=qr$ for $r\in R$ and $q\in Q$; and $\psi:P\otimes Q\to R$ is the $R$-bimodule homomorphism given by $p\otimes q\mapsto pq$, then $(P,Q,\psi)$ is an $R$-system. Recall also that $R$ is a uniquely maximal, faithful, $\psi$-compatible ideal and that if $\alpha$ is an automorphism, then $\mathcal{O}_{(P,Q,\psi)}(R)$ is isomorphic to the crossed product $R\times_\alpha\Z$ of $R$ by $\alpha$. If $R$ is unital, and we let $e=\alpha(1)$ (where 1 denotes the unit of $R$), then $e$ is an idempotent and $\alpha(R)=\alpha(R)R\alpha(R)=eRe$. It follows from \cite[Example 5.7]{CaOr2011} that we in this case have that $\mathcal{O}_{(P,Q,\psi)}(R)$ is isomorphic to the fractional skew monoid ring $R[t_+,t_-;\alpha]$ that Ara, Gonz\'alez-Barroso, Goodearl and Pardo have constructed in \cite{ArGoGo2004}. We will use these facts together with Theorem \ref{thm:simple} to give a characterization of when the crossed product $R\times_\alpha\Z$ is simple and when the fractional skew monoid ring $R[t_+,t_-;\alpha]$ is simple, but first we introduce some notions and results that we will use for this.

Unless otherwise stated, $\alpha$ will just be assumed to be an injective ring homomorphism such that $\alpha(R)R\alpha(R)\subseteq\alpha(R)$. We let $(P,Q,\psi)$ be the $R$-system defined above. Using that $R$ has local units, it is not difficult to see that for $n\in\N$, the $R$-bimodule $P^{\otimes n}$ is isomorphic to the $R$-bimodule which is equal to $\spn\{r_1\alpha^n(r_2)\mid r_1,r_2,\in R\}$ as a set, has the additive structure it inherits from $R$, and has the left and right actions given by $r\cdot p=rp$ and $p\cdot r=p\alpha^n(r)$, respectively. Likewise, $Q^{\otimes n}$ is isomorphic to the $R$-bimodule which is equal to $\spn\{\alpha^n(r_1)r_2\mid r_1,r_2\in R\}$ as a set, has the additive structure it inherits from $R$ and has the left and right given by $r\cdot q=\alpha^n(r)q$ and $q\cdot r=qr$, respectively. We will simply identify $P^{\otimes n}$ and $Q^{\otimes n}$ with these two $R$-bimodules. We will use a $\cdot$ to indicate the left and right actions of $R$ on $P^{\otimes n}$ and $Q^{\otimes n}$ to distinguish these actions from the ordinary multiplication in $R$. It is straightforward to check that if $q\in Q$, $q_n\in Q^{\otimes n}$ and $p\in P$, then $S_pT_{q_n}(q)=\alpha^n(p)\alpha(q_n)q$.
Let $(S,T,\sigma,\mathcal{O}_{(P,Q,\psi)}(R))$ denote the Cuntz-Pimsner representation of $(P,Q,\psi)$ relative to $R$. Then $S^n(p_n)\sigma(r) = S^n(p_n\alpha^n(r))$, $\sigma(r)S^n(p_n) = S^n(rp_n)$, $S^n(p_n)S^{n'}(p'_{n'})=S^{n+n'}(p_n\alpha^n(p'_{n'}))$, $T^n(q_n)\sigma(r)=T^n(q_nr)$, $\sigma(r)T^n(q_n) = T^n(\alpha^n(r)q_n)$, $T^n(q_n)T^{n'}(q'_{n'})=T^{n+n'}(\alpha^{n'}(q_n)q'_{n'})$, $S^n(p)T^n(q)=\sigma(pq)$ and $T^n(q)S^n(p)=\sigma(\alpha^{-n}(q_np_n))$ for $n,n'\in\N$, $p_n\in P^n$, $r\in R$, $n'\in P^{\otimes n'}$, $q_n\in Q^{\otimes n}$ and $q_{n'}\in Q^{\otimes n'}$ where $p_n$, $p_{n'}$, $q_n$ and $q_{n'}$ are considered as elements of $R$ and the multiplication of $R$ is used. It follows that $\mathcal{O}_{(P,Q,\psi)}(R)^{(0)}=\sigma(R)$, and that $\mathcal{O}_{(P,Q,\psi)}(R)^{(n)}=T^n(Q^{\otimes n})$ and $\mathcal{O}_{(P,Q,\psi)}(R)^{(-n)}=S^n(P^{\otimes n})$ for $n\in\N$. We say that an ideal $I$ of $R$ is strongly $\alpha$-invariant if $\alpha(I)\subseteq I$ and $\alpha(R)I\alpha(R)\subseteq \alpha(I)$ (this is equivalent to $\alpha(R)I\alpha(R)= \alpha(I)$ since $R$ has local units).

\begin{prop}
	Let $R$ be a ring with local units, $\alpha:R\to R$ an injective ring homomorphism satisfying $\alpha(R)R\alpha(R)\subseteq\alpha(R)$, and let $(P,Q,\psi)$ be the $R$-system defined above. Then there is a bijective correspondence between graded ideals of $\mathcal{O}_{(P,Q,\psi)}(R)$ and strongly $\alpha$-invariant ideals of $R$. 
\end{prop}

\begin{proof}	
	For each strongly $\alpha$-invariant ideal $I$ in $R$, let $H_I$ be the ideal in $\mathcal{O}_{(P,Q,\psi)}(R)$ generated by $\sigma(I)$; and let for each graded ideal $H$ in $\mathcal{O}_{(P,Q,\psi)}(R)$, $I_H=\{x\in R\mid \sigma(x)\in H\}$. We will show that $H_I$ is a graded ideal in $\mathcal{O}_{(P,Q,\psi)}(R)$, that $I_H$ is a strongly $\alpha$-invariant ideal in $R$, and that $I_{H_I}=I$ and $H_{I_H}=H$ for all strongly $\alpha$-invariant ideals $I$ in $R$ and all graded ideals $H$ in $\mathcal{O}_{(P,Q,\psi)}(R)$. This will establish the bijective correspondence between the graded ideals of $\mathcal{O}_{(P,Q,\psi)}(R)$ and the strongly $\alpha$-invariant ideals of $R$. 
	
	Let $I$ be a strongly $\alpha$-invariant ideal in $R$. It is not difficult to check that if we let $H^{(0)}=\sigma(I)$ and for each $n\in\N$ let $H^{(n)}=\spn\{T^n(\alpha^n(r)x)\mid r\in R,\ x\in I\}$ and $H^{(-n)}=\spn\{S^n(x\alpha^n(r))\mid x\in I,\ r\in R\}$, then $\oplus_{n\in\Z}H^{(n)}$ is an ideal in $\mathcal{O}_{(P,Q,\psi)}(R)$. Since $\oplus_{n\in\Z}H^{(n)}$ contains $\sigma(I)$ and itself must be contained in any ideal which contains $\sigma(I)$, it must be the case that $H_I=\oplus_{n\in\Z}H^{(n)}$. It follows that $H_I$ is graded and that $I_{H_I}=I$.
	
	Let $H$ be a graded ideal in $\mathcal{O}_{(P,Q,\psi)}(R)$. It is clear that $I_H$ is an ideal in $R$. Assume that $x\in I_H$. Choose idempotens $e_1,e_2\in R$ such that $e_1\alpha(x)e_1=\alpha(x)$ and $e_2xe_2=x$. Then 
	\begin{equation*}
		\sigma(\alpha(x))=S(e_1\alpha(e_2))\sigma(x)T(\alpha(e_2)e_1)\in H, 
	\end{equation*}
	so $\alpha(x)\in I_H$. Assume then that $r_1,r_2\in R$. Choose idempotents $f_1,f_2\in R$ such that $f\alpha(r_1)f_1=\alpha(r_1)$ and $f_2\alpha(r_2)=\alpha(r_2)$. Then 
	\begin{equation*}
		\sigma(\alpha^{-1}(\alpha(r_1)x\alpha(r_2)))= T(\alpha(r_1)f_1)\sigma(x)S(f_2\alpha(r_2))\in H, 
	\end{equation*}
	so $\alpha(r_1)x\alpha(r_2)\in \alpha(I_H)$. This shows that $I_H$ is a strongly $\alpha$-invariant ideal in $R$. Since $\mathcal{O}_{(P,Q,\psi)}(R)^{(0)}=\sigma(R)$, it follows from \cite[Lemma 3.35]{CaOr2011} that $H$ is generated by $\sigma(I_H)$. Thus $H=H_{I_H}$.
\end{proof}

By combining the above result with Remark \ref{remark:supermax} we get the following characterization of when $R$ is a super maximal $\psi$-compatible ideal.

\begin{cor}\label{cor:supermax}
	Let $R$ be a ring with local units, $\alpha:R\to R$ an injective ring homomorphism satisfying $\alpha(R)R\alpha(R)\subseteq\alpha(R)$, and let $(P,Q,\psi)$ be the $R$-system defined above. Then the following three conditions are equivalent:
	\begin{enumerate}
		\item The ring $R$ is a super maximal $\psi$-compatible ideal.
		\item The only graded ideals in $\mathcal{O}_{(P,Q,\psi)}(R)$ are $\{0\}$ and $\mathcal{O}_{(P,Q,\psi)}(R)$.
		\item The only strongly $\alpha$-invariant ideals in $R$ are $\{0\}$ and $R$.
	\end{enumerate}
\end{cor}

We next introduce the \emph{multiplier ring} of $R$ (see for example \cite{ArPe2000}). A double centralizer on $R$ is a pair $(f,g)$ where $f:R\to R$ is a right $R$-module homomorphism and $g:R\to R$ is a left $R$-module homomorphism satisfying $r_1f(r_2)=g(r_1)r_2$ for all $r_1,r_2\in R$. The \emph{multiplier ring} of $R$ is the ring $\mathcal{M}(R)$ of all double centralizers on $R$ with addition defined by $(f_1,g_1)+(f_2,g_2)=(f_1+f_2,g_1+g_2)$ and product defined by $(f_1,g_1)(f_2,g_2)=(f_1\circ f_2,g_2\circ g_1)$. Notice that $(\id_R,\id_R)$ is a unit of $\mathcal{M}(R)$. There is a ring homomorphism $\iota:R\to\mathcal{M}(R)$ given by $\iota(r)=(f_r,g_r)$ where $f_r(s)=rs$ and $g_r(s)=sr$ for $r,s\in R$. Since $R$ has local units, $\iota$ is injective. We will therefore simple regard $R$ as a subring of $\mathcal{M}(R)$. We then have that if $u=(f,g)\in\mathcal{M}(R)$ and $r\in R$, then $ur=f(r)$ and $ru=g(r)$. It follows that $R$ is an ideal in $\mathcal{M}(R)$. Notice that $R=\mathcal{M}(R)$ if and only if $R$ is unital.

\begin{definition}\label{def:inner}
	Let $n\in\N$ and let $R$ be a ring with local units. A ring homomorphism $\alpha:R\to R$ is said to be \emph{inner with periodicity $n$} if there exist $u,v\in\mathcal{M}(R)$ such that $vu=1$ (where $1$ denotes the unit of $\mathcal{M}(R)$), and $\alpha^n(r)=urv$ and $\alpha(ur)=u\alpha(r)$ for all $r\in R$. If $\alpha$ is not inner of any periodicity, then it is said to be \emph{outer}.
\end{definition}

\begin{remark}\label{remark:inner}
	Notice that if $\alpha$ is an automorphism and $u,v$ are as above, then $v$ is the inverse of $u$.
\end{remark}

In \cite{ArPe2000} the authors introduce a topology on $\mathcal{M}(R)$ in the following way. A net $(x_\lambda)_{\lambda\in \Lambda}$ of elements of $\mathcal{M}(R)$ converges \textit{strictly} to an a element $x\in \mathcal{M}(R)$ if there for every $r\in R$ exists $\lambda_0\in \Lambda$ such that $(x_\lambda-x)r=r(x_\lambda-x)=0$ for $\lambda\geq \lambda_0$. Since $R$ has local units, a net in $\mathcal{M}(R)$ can at most converges strictly to one element. Such an element will, if it exists, be called the \emph{strict limit} of the net. A net $(x_\lambda)_{\lambda\in \Lambda}$ is \emph{Cauchy} if there for every $r\in R$ exists $\lambda_0\in \Lambda$ such that $r(x_\lambda-x_\mu)=(x_\lambda-x_\mu)r=0$ for $\lambda,\mu\geq \lambda_0$. 
It is shown in \cite[Proposition 1.6]{ArPe2000} that if $R$ has local units, then every Cauchy net in $\mathcal{M}(R)$ converges strictly, and that every element of $\mathcal{M}(R)$ is the strict limit of a net of elements of $R$.

A net $(r_\lambda)_{\lambda\in \Lambda}$ of elements of $R$ that converges to the unit of $\mathcal{M}(R)$ is called an \emph{approximate unit} for $R$. Notice that in case $R$ has local units we can construct an approximate unit $(e_\lambda)_{\lambda\in \Lambda}$ consisting of idempotents simple by letting $\Lambda$ be the directed set of finite subsets of $R$ ordered by inclusion, and then for every $\lambda\in\Lambda$ choosing an idempotent $e_\lambda$ such that $e_\lambda r=re_\lambda=r$ for every $r\in \lambda$. 

\begin{definition}\label{def:strict}
Let $R$ be a ring with local units. A ring homomorphism $\alpha:R\to R$ is said to be \emph{strict} if there exists an approximate unit $(e_\lambda)_{\lambda\in\Lambda}$ for $R$ consisting of idempotents such that $(\alpha(e_\lambda))_{\lambda\in\Lambda}$ converges strictly.
\end{definition}

\begin{remark}\label{remark:strict}
	Notice that if $\alpha$ is an automorphism, then it is strict (since $(\alpha(e_\lambda))_{\lambda\in\Lambda}$ converges strictly to the unit in that case). Notice also that if $R$ is unital, then every ring homomorphism $\alpha:R\to R$ is automatically strict (because the net consisting of just $1$ is an approximate unit in that case). 
\end{remark}

\begin{prop} \label{prop:inner}
	Let $R$ be a ring with local units, $\alpha:R\to R$ an injective ring homomorphism satisfying $\alpha(R)R\alpha(R)\subseteq\alpha(R)$, and let $(P,Q,\psi)$ be the $R$-system defined above. Consider the following three conditions:
	\begin{enumerate}
		\item \label{item:221}  There exists an $n\in\N$ such that the homomorphism $\alpha$ is inner with periodicity $n$.
		\item \label{item:222}  The ring $R$ is a $\psi$-invariant cycle.
		\item\label{item:223} The ring $R$ does not satisfy condition (L) with respect to $(P,Q,\psi)$.
	\end{enumerate}
	Then \eqref{item:221} implies \eqref{item:222}, and \eqref{item:222} implies \eqref{item:223}. If in addition $R$ is a super maximal $\psi$-compatible ideal, and $\alpha^n$ is strict for every $n\in\N$, then \eqref{item:223} implies \eqref{item:221} and the three conditions are equivalent.
\end{prop}

\begin{proof}
	$\eqref{item:221}\implies\eqref{item:222}$:
	Let $u$ and $v$ be elements in $\mathcal{M}(R)$ such that $vu=1$, and $urv=\alpha^n(r)$ and $\alpha(ux)=u\alpha(x)$ for all $r\in R$. Define $\eta:R\to R$ by $\eta(r)=ur$. Let $r\in R$. Choose $e\in R$ such that $er=r$. Then we have that $\eta(r)=ur=uer=uevur=\alpha^n(e)ur$. This shows that $\eta(R)\subseteq Q^{\otimes n}$. It is clear that $\eta$ is additive and injective. Let $r_1,r_2,\in R$. Then $\eta(r_1r_2)=ur_1r_2=\eta(r_1)r_2$ and $\eta(r_1r_2)=ur_1r_2=\alpha^n(r_1)ur_2=\alpha^n(r_1)\eta(r_2)$, which shows that $\eta$ is an $R$-bimodule homomorphism from $R$ to $Q^{\otimes n}$. Let $p\in P$, $r\in R$ and $q\in Q$. Then we have that
	\begin{equation*}
		\begin{split}
			\eta\left(\psi(p\cdot r\otimes q)\right)
			&=\eta\left(p\alpha(r)q\right)=up\alpha(r)q=\alpha^n(p)u\alpha(r)q\\
			&=\alpha^n(p)\alpha(ur)q=\alpha^n(p)\alpha(\eta(r))q=S_pT_{\eta(r)}(q).
		\end{split}
	\end{equation*}
	 Thus $R$ is a $\psi$-invariant cycle.
	
	$\eqref{item:222}\implies\eqref{item:223}$:
	It is easy to see that $\psi^{-1}(R)=R$ from which it follows that $R^{[\infty]}=R$. Thus, if $R$ is a $\psi$-invariant cycle, then $R$ does not satisfy condition (L) with respect to $(P,Q,\psi)$.
	
	$\eqref{item:223}\implies\eqref{item:221}$: 
	Assume that $R$ does not satisfy condition (L) with respect to $(P,Q,\psi)$. It then follows from Proposition \ref{prop:main} that there is a non-zero graded ideal $\bigoplus_{k\in\Z}H^{(k)}$ in $\mathcal{O}_{(P,Q,\psi)}(R)$, an $n\in\N$ and a family $(\phi_k)_{k\in\Z}$ of injective $\mathcal{O}_{(P,Q,\psi)}(R)^{(0)}$-bimodule homomorphisms $\phi_k:H^{(k)}\to\mathcal{O}_{(P,Q,\psi)}(R)^{(k+n)}$ such that $x\phi_k(y)=\phi_{k+j}(xy)$ and $\phi_k(y)x=\phi_{k+j}(yx)$ for $k,j\in\Z$, $x\in\mathcal{O}_{(P,Q,\psi)}(R)^{(j)}$ and $y\in H^{(k)}$. Notice that also $\bigoplus_{k\in\Z}\phi_{k-n}(H^{(k-n)})$ is a non-zero graded ideal in $\mathcal{O}_{(P,Q,\psi)}(R)$. If $R$ is a super maximal $\psi$-compatible ideal, then it follows from Corollary \ref{cor:supermax} that $\bigoplus_{k\in\Z}H^{(k)}=\bigoplus_{k\in\Z}\phi_{k-n}(H^{(k-n)})=\mathcal{O}_{(P,Q,\psi)}(R)$ from which it follows that $H^{(0)}=\phi_{-n}(H^{(-n)})=\mathcal{O}_{(P,Q,\psi)}(R)^{(0)}=\sigma(R)$, $\phi_0(H^{(0)})=\mathcal{O}_{(P,Q,\psi)}(R)^{(n)}=T^n(Q^{\otimes n})$ and $H^{(-n)}=\mathcal{O}_{(P,Q,\psi)}(R)^{(-n)}=S^n(P^{\otimes n})$. Suppose in addition that $\alpha^n$ is strict, and let $(e_\lambda)_{\lambda\in\Lambda}$ be an approximate unit for $R$ consisting of idempotents such that $(\alpha(e_\lambda))_{\lambda\in\Lambda}$ converges strictly. Since $T^n$ and $\phi_{-n}$ are injective, and $Q^{\otimes n}$ and $P^{\otimes n}$ are subsets of $R$, there exists for each $\lambda\in\Lambda$ a unique $u_\lambda\in R$ such that $T^n(u_\lambda)=\phi_0(\sigma(e_\lambda))$ and a unique $v_\lambda\in R$ such that $\phi_{-n}(S^n(v_\lambda))=\sigma(e_\lambda)$. Notice that
	\begin{equation*}
		T^n(u_\lambda)=\phi_0(\sigma(e_\lambda))
		=\phi_0(\sigma(e_\lambda e_\lambda))
		=\sigma(e_\lambda)\phi_0(\sigma(e_\lambda))
		=\sigma(e_\lambda)T^n(u_\lambda)=T^n(\alpha^n(e_\lambda)u_\lambda).
	\end{equation*}
	It follows that $\alpha^n(e_\lambda)u_\lambda=u_\lambda$. If $\lambda,\lambda_1\in\Lambda$ and $e_{\lambda_1}e_\lambda=e_{\lambda_1}$, then
	\begin{equation*}
		\begin{split}
		T^n(\alpha^n(e_{\lambda_1})u_\lambda)
		&=\sigma(e_{\lambda_1})T^n(u_\lambda)
		=\sigma(e_{\lambda_1})\phi_0(\sigma(e_\lambda))\\
		&=\phi_0(\sigma(e_{\lambda_1}e_\lambda))
		=\phi_0(\sigma(e_{\lambda_1}))
		=T^n(u_{\lambda_1}),
		\end{split}
	\end{equation*}
	from which it follows that $\alpha^n(e_{\lambda_1})u_\lambda=u_{\lambda_1}$. Let $r\in R$. Choose $\lambda_1,\lambda_2,\lambda_3\in\Lambda$ such that $r\alpha^n(e_\lambda)=r\alpha^n(e_{\lambda_1})$ for $\lambda\ge\lambda_1$, $e_{\lambda_1}e_\lambda=e_{\lambda_1}$ for $\lambda\ge\lambda_2$, and $e_\lambda r=r$ for $\lambda\ge \lambda_3$. If $\lambda\ge \lambda_1,\lambda_2,\lambda_3$, then 
	\begin{equation*}
		ru_\lambda=r\alpha^n(e_\lambda)u_\lambda=r\alpha^n(e_{\lambda_1})u_\lambda
		=ru_{\lambda_1},
	\end{equation*}
	and
	\begin{equation*}
		T^n(u_\lambda r)
		=T^n(u_\lambda)\sigma(r)=\phi_0(\sigma(e_\lambda))\sigma(r)
		=\phi_0(\sigma(e_\lambda r))=\phi_0(\sigma(r)).
	\end{equation*}
	This shows that $(u_\lambda)_{\lambda\in\Lambda}$ is Cauchy and hence converges strictly to an element $u\in\mathcal{M}(R)$. One can by a similar method show that $(v_\lambda)_{\lambda\in\Lambda}$ converges strictly to an element $v\in\mathcal{M}(R)$.
	
	Let $\lambda\in\Lambda$. Then 
	\begin{equation*}
		\begin{split}
		\sigma(v_\lambda u_\lambda)&=S^n(v_\lambda)T^n(u_\lambda)
		=S^n(v_\lambda)\phi_0(\sigma(e_\lambda))
		=\phi_{-n}(S^n(v_\lambda)\sigma(e_\lambda))\\
		&=\phi_{-n}(S^n(v_\lambda))\sigma(e_\lambda)
		=\sigma(e_\lambda)\sigma(e_\lambda)=\sigma(e_\lambda),
	\end{split}
	\end{equation*}
	from which it follow that $v_\lambda u_\lambda=e_\lambda$. Thus $vu=1$.
	
	Let $r\in R$. Choose $\lambda_0\in\Lambda$ such that $re_\lambda=e_\lambda r=r$ for $\lambda\ge\lambda_0$. If $\lambda\ge\lambda_0$, then 
	\begin{equation*}
		T^n(\alpha^n(r)u_\lambda)=\sigma(r)\phi_0(\sigma(e_\lambda))
		=\phi_0(\sigma(re_\lambda))=\phi_0(\sigma(e_\lambda r))
		=\phi_0(\sigma(e_\lambda))\sigma(r)=T^n(u_\lambda r).
	\end{equation*}
	It follows that $\alpha^n(r)u=ur$ and thus that $urv=\alpha^n(r)$.
	
	Let $r\in R$. Choose $\lambda_0\in\Lambda$ such that $e_\lambda r=r$ and $e_\lambda\alpha(r)=\alpha(r)$ for $\lambda\ge\lambda_0$. If $\lambda\ge\lambda_0$ then
	\begin{equation*}
		\begin{split}
		T^n(\alpha(u_\lambda r))
		&=T^n\bigl(\alpha^{n+1}(e_\lambda)\alpha(u_\lambda)\alpha(r)\bigr)
		=S(\alpha(e_\lambda))T^n(u_\lambda) T(\alpha(r))\\
		&=S(\alpha(e_\lambda))\phi_0(\sigma(e_\lambda))T(\alpha(r))
		=	\phi_0\bigl(S(\alpha(e_\lambda))\bigr) \sigma(e_\lambda)T(\alpha(r))\\
		&=\phi_0\bigl(\sigma(\alpha(e_\lambda e_\lambda r))\bigr)
		=\phi_0\bigl(\sigma(\alpha(r))\bigr)
		=\phi_0\bigl(\sigma(e_\lambda\alpha(r))\bigr)
		=\phi_0(\sigma(e_\lambda))\sigma(\alpha(r))\\
		&=T^n(u_\lambda)\sigma(\alpha(r))=T^n(u_\lambda\alpha(r)),
		\end{split}
	\end{equation*}
	from which it follows that $\alpha(u_\lambda r)=u_\lambda\alpha(r)$. Thus $\alpha(ur)=u\alpha(r)$. 
	
	Hence $\alpha$ is inner with periodicity $n$ in this case.
\end{proof}

By combining Theorem \ref{thm:simple} and Corollary \ref{cor:supermax} with Remark \ref{remark:strict}, Proposition \ref{prop:inner}, and the fact that $\mathcal{O}_{(P,Q,\psi)}(R)$ is isomorphic to the crossed product $R\times_\alpha\Z$ of $R$ by $\alpha$ when $\alpha$ is an automorphism, and to the fractional skew monoid ring $R[t_+,t_-;\alpha]$ when $R$ is unital and $\alpha$ is an injective homomorphism such that $\alpha(R)=eRe$ for some idempotent $e\in R$, we get the following two corollaries.

\begin{cor} \label{cor:fractional}
	Let $R$ be a unital ring and let $\alpha:R\to R$ be an injective ring homomorphism such that $\alpha(R)=eRe$ for some idempotent $e\in R$. Then the following two statements are equivalent: 
	\begin{enumerate}
		\item The fractional skew monoid ring $R[t_+,t_-;\alpha]$ is simple. 
		\item The homomorphism $\alpha$ is outer and the only strongly $\alpha$-invariant ideals in $R$ are $\{0\}$ and $R$. 
	\end{enumerate}
\end{cor}

\begin{cor} \label{cor:auto}
	Let $R$ be a ring with local units and let $\alpha:R\to R$ be a ring automorphism. Then the following two statements are equivalent: 
	\begin{enumerate}
		\item  The crossed product $R\times_\alpha\Z$ is simple. 
		\item  The automorphism $\alpha$ is outer and the only strongly $\alpha$-invariant ideals in $R$ are $\{0\}$ and $R$. 
	\end{enumerate}
\end{cor}

We end by noticing that when $\alpha$ is an automorphism, the condition of $\alpha$ being outer is equivalent with the seemingly stronger, and perhaps more familiar, condition that $\alpha$ is \emph{strongly outer}.

\begin{definition}
	Let $n\in\N$ and let $R$ be a ring with local units and $\alpha:R\to R$ a ring automorphism. If there exists an invertible element $u\in\mathcal{M}(R)$ such that $\alpha^n(r)=uru^{-1}$ for all $r\in R$, then $\alpha$ is said to be \emph{weakly inner with periodicity $n$}. If $\alpha$ is not weakly inner of any periodicity, then it is said to be \emph{strongly outer}.
\end{definition} 

\begin{prop}
	Let $R$ be a ring with local units and let $\alpha:R\to R$ a ring automorphism. Then $\alpha$ is outer if and only if it is strongly outer.
\end{prop}

\begin{proof}
	It follows from Remark \ref{remark:inner} that if $\alpha$ is strongly outer, then it is also outer.
	
	Suppose that $\alpha$ is not strongly outer. Then there exist $n\in\N$ and an invertible element $u\in\mathcal{M}(R)$ such that $\alpha^n(r)=uru^{-1}$ for all $r\in R$. If $x=(f,g)\in\mathcal{M}(R)$ where $(f,g)$ is a double centralizer, then we let $\hat{\alpha}(x)$ denote the double centralizer $(\alpha\circ f\circ\alpha^{-1},\alpha\circ g\circ\alpha^{-1})$. It is easy to check that $x\mapsto\hat{\alpha}(x)$ defines an automorphism $\hat{\alpha}$ of $\mathcal{M}(R)$ and that $\hat{\alpha}^n(x)=uxu^{-1}$ for all $x\in\mathcal{M}(R)$. In particular $\hat{\alpha}^n(u)=uuu^{-1}=u$ and $\hat{\alpha}^n(u^{-1})=uu^{-1}u^{-1}=u^{-1}$. Let 
	\begin{equation*}
		u'=u\hat{\alpha}(u)\dots\hat{\alpha}^{n-1}(u)\text{ and } v'=\hat{\alpha}^{n-1}(u^{-1})\dots\hat{\alpha}(u^{-1})u^{-1}.
	\end{equation*}
	Then $v'u'=1$. If $r\in R$, then
	\begin{equation*}
		\begin{split}
		\alpha(u'r)&=\hat{\alpha}(u')\alpha(r)=\hat{\alpha}
		\bigl(u\hat{\alpha}(u)\dots\hat{\alpha}^{n-1}(u)\bigr)\alpha(r)	
		=\hat{\alpha}(u)\hat{\alpha}^2(u)\dots\hat{\alpha}^n(u)\alpha(r)\\
		&=\hat{\alpha}^{n+1}(u)\hat{\alpha}^{n+2}(u)\dots\hat{\alpha}^{2n}(u)\alpha(r)
		=\hat{\alpha}^n
		\bigl(\hat{\alpha}(u)\hat{\alpha}^2(u)\dots\hat{\alpha}^n(u)\bigr)\alpha(r)\\
		&=u\hat{\alpha}(u)\hat{\alpha}^2(u)\dots\hat{\alpha}^n(u)u^{-1}\alpha(r)
		=u\hat{\alpha}(u)\hat{\alpha}^2(u)\dots\hat{\alpha}^{n-1}(u)uu^{-1}\alpha(r)
		=u'\alpha(r)	
	\end{split}
	\end{equation*}
	and
	\begin{equation*}
		\begin{split}
		u'rv'&=u\hat{\alpha}(u)\dots\hat{\alpha}^{n-1}(u)r
		\hat{\alpha}^{n-1}(u^{-1})\dots\hat{\alpha}(u^{-1})u^{-1}\\
		&=\hat{\alpha}^n\bigl(\hat{\alpha}(u)\dots\hat{\alpha}^{n-1}(u)r
		\hat{\alpha}^{n-1}(u^{-1})\dots\hat{\alpha}(u^{-1})\bigr)\\
		&=\hat{\alpha}(u)\dots\hat{\alpha}^{n-1}(u)\alpha^n(r)
		\hat{\alpha}^{n-1}(u^{-1})\dots\hat{\alpha}(u^{-1})\\
		&=\hat{\alpha}(u)\dots\hat{\alpha}^{n-2}(u)\alpha^{n+1}(r)
		\hat{\alpha}^{n-2}(u^{-1})\dots\hat{\alpha}(u^{-1})\\
		&\phantom{=\hat{\alpha}(u)\alpha^{(n-1)n}(r)}\vdots\\
		&=\hat{\alpha}(u)\alpha^{(n-1)n}(r)\hat{\alpha}(u^{-1})=\alpha^{n^2}(r).
	\end{split}
	\end{equation*}
	Thus $\alpha$ is inner with periodicity $n^2$ and is therefore not outer.
\end{proof}
\section*{Acknowledgments}

Part of this work was done during visits of the third author to the Institut for Matematik og Datalogi, Syddansk Universitet and to the Institut for Matematiske Fag, K\o benhavns Universitet (Denmark). The third author thanks both host centers for their kind hospitality.


\begin{thebibliography}{10}

\bibitem{AbAr2005}
Gene Abrams and Gonzalo Aranda~Pino.
\newblock The {L}eavitt path algebra of a graph.
\newblock {\em J. Algebra}, 293(2):319--334, 2005.

\bibitem{AbAr2008}
Gene Abrams and Gonzalo Aranda~Pino.
\newblock The {L}eavitt path algebras of arbitrary graphs.
\newblock {\em Houston J. Math.}, 34(2):423--442, 2008.

\bibitem{ArGoGo2004}
P.~Ara, M.~A. Gonz{\'a}lez-Barroso, K.~R. Goodearl, and E.~Pardo.
\newblock Fractional skew monoid rings.
\newblock {\em J. Algebra}, 278(1):104--126, 2004.

\bibitem{ArPe2000}
Pere Ara and Francesc Perera.
\newblock Multipliers of von {N}eumann regular rings.
\newblock {\em Comm. Algebra}, 28(7):3359--3385, 2000.

\bibitem{CaOr2011}
Toke~Meier Carlsen and Eduard Ortega.
\newblock Algebraic {C}untz--{P}imsner rings.
\newblock {\em Proc. London Math. Soc.}, doi: 10.1112/plms/pdq040, 2011.

\bibitem{FoRa1999}
Neal~J. Fowler and Iain Raeburn.
\newblock The {T}oeplitz algebra of a {H}ilbert bimodule.
\newblock {\em Indiana Univ. Math. J.}, 48(1):155--181, 1999.

\bibitem{Ka2004a}
Takeshi Katsura.
\newblock On {$C\sp *$}-algebras associated with {$C\sp *$}-correspondences.
\newblock {\em J. Funct. Anal.}, 217(2):366--401, 2004.

\bibitem{Ka2006b}
Takeshi Katsura.
\newblock A class of {$C\sp *$}-algebras generalizing both graph algebras and
  homeomorphism {$C\sp *$}-algebras. {III}. {I}deal structures.
\newblock {\em Ergodic Theory Dynam. Systems}, 26(6):1805--1854, 2006.

\bibitem{Ka2007}
Takeshi Katsura.
\newblock Ideal structure of {$C\sp *$}-algebras associated with {$C\sp
  *$}-correspondences.
\newblock {\em Pacific J. Math.}, 230(1):107--145, 2007.

\bibitem{Mo1980}
Susan Montgomery.
\newblock {\em Fixed rings of finite automorphism groups of associative rings},
  volume 818 of {\em Lecture Notes in Mathematics}.
\newblock Springer, Berlin, 1980.

\bibitem{MuSo1998}
Paul~S. Muhly and Baruch Solel.
\newblock Tensor algebras over {$C\sp *$}-correspondences: representations,
  dilations, and {$C\sp *$}-envelopes.
\newblock {\em J. Funct. Anal.}, 158(2):389--457, 1998.

\bibitem{Pa1989}
Donald~S. Passman.
\newblock {\em Infinite crossed products}, volume 135 of {\em Pure and Applied
  Mathematics}.
\newblock Academic Press Inc., Boston, MA, 1989.

\bibitem{Pi1997}
Michael~V. Pimsner.
\newblock A class of {$C\sp *$}-algebras generalizing both {C}untz-{K}rieger
  algebras and crossed products by {${\bf Z}$}.
\newblock In {\em Free probability theory (Waterloo, ON, 1995)}, volume~12 of
  {\em Fields Inst. Commun.}, pages 189--212. Amer. Math. Soc., Providence, RI,
  1997.

\bibitem{Ra2005}
Iain Raeburn.
\newblock {\em Graph algebras}, volume 103 of {\em CBMS Regional Conference
  Series in Mathematics}.
\newblock Published for the Conference Board of the Mathematical Sciences,
  Washington, DC, 2005.

\bibitem{Si2006}
Aidan Sims.
\newblock Relative {C}untz-{K}rieger algebras of finitely aligned higher-rank
  graphs.
\newblock {\em Indiana Univ. Math. J.}, 55(2):849--868, 2006.

\bibitem{To2003d}
Mark Tomforde.
\newblock Simplicity of ultragraph algebras.
\newblock {\em Indiana Univ. Math. J.}, 52(4):901--925, 2003.

\bibitem{To2007}
Mark Tomforde.
\newblock Uniqueness theorems and ideal structure for {L}eavitt path algebras.
\newblock {\em J. Algebra}, 318(1):270--299, 2007.

\end{thebibliography}
\end{document}